\documentclass{hha}

\usepackage{graphicx, paralist, indentfirst, amscd} 

\theoremstyle{plain}
\newtheorem{prop}[theorem]{Proposition}
\newtheorem{thm}[theorem]{Theorem}
\newtheorem{lem}[theorem]{Lemma}

\theoremstyle{remark}
\newtheorem*{rem*}{Remark}

\begin{document}

\newcommand{\bek}{\quad \\[-6pt] \indent}

\newcommand{\Ker}{\mathop{\mathrm{Ker}}\nolimits}
\newcommand{\rk}{\mathop{\mathrm{rk}}\nolimits}
\newcommand{\id}{\mathop{\mathrm{id}}\nolimits}
\newcommand{\pr}{\mathop{\mathrm{pr}}\nolimits}
\newcommand{\DIV}{\mathop{\mathrm{div}}\nolimits}
\newcommand{\interior}{\mathop{\mathrm{int}}\nolimits}
\newcommand{\Tor}{\mathop{\mathrm{Tor}}\nolimits}
\newcommand{\Tang}{\mathcal{T}}
\newcommand{\Homeo}{\mathop{\mathrm{Homeo}}\nolimits}
\newcommand{\Z}{\mathbb{Z}}
\newcommand{\C}{\mathbb{C}}
\newcommand{\R}{\mathbb{R}}
\newcommand{\T}{\mathbf{T}}
\renewcommand{\S}{\mathbf{S}}

\title{Sum of embedded submanifolds} 


\author{Csaba Nagy}             

\email{cnagy@student.unimelb.edu.au}

%
%
\address{School of Mathematics and Statistics,
         The University of Melbourne,
         Parkville, VIC, 3000,
         Australia}







\begin{abstract}
In an $n$-manifold $X$ each element of $H_{n-1}(X; \Z_2)$ can be represented by an embedded codimension-1 submanifold. Hence for any two such submanifolds there is a third one that represents the sum of their homology classes. We construct such a representative explicitly. We describe the analogous construction for codimension-2 co-oriented submanifolds, and examine the special case of oriented and/or co-oriented submanifolds. We also give a lower bound for the number of connected components of the intersection of two oriented codimension-1 submanifolds in terms of the homology classes they represent.
\end{abstract}


\maketitle


\section{Introduction}

Let $X$ be a smooth ($C^{\infty}$) closed $n$-dimensional manifold, and let $Y_1$ and $Y_2$ be two smooth closed embedded submanifolds in $X$ that intersect each other transversally. We want to construct an embedded submanifold $Y$ that is the ``sum" of $Y_1$ and $Y_2$. By this we mean that if $Y_1$ and $Y_2$ represent homology classes $[Y_1]$ and $[Y_2]$, then $Y$ should represent $[Y_1]+[Y_2]$. More specifically, we want $Y$ to be an embedded approximation of $Y_1 \cup Y_2$, in the sense that it coincides with $Y_1 \cup Y_2$ outside a neighbourhood of $Y_1 \cap Y_2$ (this implies that $[Y] = [Y_1]+[Y_2]$, see Lemma \ref{lem:w} below).

\bek
Recall that a submanifold is co-oriented if its normal bundle is oriented. We will consider two cases: \\
\textbf{Case 1.} $Y_1$ and $Y_2$ are codimension-1 submanifolds \\
\textbf{Case 2.} $Y_1$ and $Y_2$ are codimension-2 co-oriented submanifolds 

\bek
In these cases the existence of a $Y$ with $[Y] = [Y_1]+[Y_2]$ follows immediately from a classical result of Thom ([3]). For Case 1 we use that all homology classes in $H_{n-1}(X; \Z_2)$ can be represented by submanifolds: By the Pontryagin-Thom construction there is a bijection between the set of codimension-$1$ embedded submanifolds (up to cobordism) and $[X, {\R}P^{\infty}]$, the set of homotopy classes of maps $X \rightarrow {\R}P^{\infty}$. Since ${\R}P^{\infty}$ is a $K(\Z_2, 1)$ Eilenberg-MacLane-space, this set is a group, and it is isomorphic to $H^1(X; \Z_2)$. Finally by Poincar\'e-duality $H^1(X; \Z_2) \cong H_{n-1}(X; \Z_2)$. The composition of these bijections maps (the cobordism class of) an embedded submanifold to the homology class it represents.

The situation is similar for codimension-2 co-oriented submanifolds: their cobordism classes are in bijection with $[X, {\C}P^{\infty}] \cong H^2(X; \Z) \cong H_{n-2}(X; \Z_w)$ ($\Z_w$ de\-notes local coefficients in the orientation $\Z$-bundle of $X$).

In fact, even the existence of an embedded approximation $Y$ of $Y_1 \cup Y_2$ can be shown easily. Clearly, if $Y_1$ and $Y_2$ are disjoint, then $Y$ can be chosen to be $Y_1 \cup Y_2$. If they intersect each other, then a transversality argument is needed. There are $\R^1$-bundles ($\C^1$-bundles in Case 2) $\eta_1$, $\eta_2$ over $X$, and sections $s_1$, $s_2$ thereof such that $s_i$ is zero exactly in the points of $Y_i$. Then $s_1 \otimes s_2$ is a section of $\eta_1 \otimes \eta_2$, and if we make it transversal to the zero-section, then its zeros will form a suitable $Y$.

So it is known that an embedded approximation $Y$ exists, but the proof of this is not constructive. It was the question of Matthias Kreck whether a $Y$ could be constructed explicitly from $Y_1$ and $Y_2$. In this paper we answer this question by describing such a construction.

We also consider additional requirements of orientability or co-orientability. In the case of oriented codimension-1 submanifolds this construction, combined with the results of Meeks--Patrusky ([1]) and Meeks ([2]), allows us to give a lower bound for the number of connected components of $Y_1 \cap Y_2$ in terms of the homology classes $[Y_1],[Y_2]$ (see Theorems \ref{thm:lb1} and \ref{thm:lb2}).

\subsection*{Overview of the construction}

We will formulate most of our statements for Case 2, the appropriate statements for Case 1 can be obtained by replacing $U(1)$, $\C$, $D^4$, $D^2$ and $S^1$ with $O(1)$, $\R$, $D^2$, $D^1$ and $S^0$ respectively (and sometimes a few other changes are needed, these will be indicated in \big[ \big] brackets).

Let $M = Y_1 \cap Y_2$, it is a codimension-4 \big[codimension-2\big] submanifold in $X$. Let $T$ be a tubular neighbourhood of $M$ in $X$. Then $T$ is diffeomorphic to the total space of a smooth $\left( U(1) \times U(1) \right)$-bundle over $M$ with fiber $D^4 \approx D^2 \times D^2$, and the action of the structure group on the fiber is given by $(\alpha, \beta)(x,y)=(\alpha(x), \beta(y))$, for all $(\alpha, \beta) \in U(1) \times U(1)$, $(x,y) \in D^2 \times D^2$. (The normal bundle of $M$ is the Whitney sum of the normal budles of $Y_1$ and $Y_2$, restricted to $M$.) This bundle $T \rightarrow M$ will be denoted by $\T$. 

The submanifolds $Y_1 \cap T \subset T$ and $Y_2 \cap T \subset T$ are the total spaces of subbundles of $\T$ with fiber $D^2 \times \{ 0 \}$ and  $\{ 0 \} \times D^2$ respectively. Their boundaries, $Y_i \cap \partial T$, will be denoted by $B_i$, these are subbundles with fiber $S^1 \times \{ 0 \}$ and $\{ 0 \} \times S^1$. Let $B = B_1 \bigsqcup B_2$. 

\begin{lem} \label{lem:w}
If $W \subset T$ is a co-oriented embedded submanifold with boundary $B$ (and its co-orientation agrees with that of $Y_1 \cup Y_2$ over $B$), and $Y = ((Y_1 \cup Y_2) \setminus T) \cup W$, then $[Y] = [Y_1]+[Y_2]$.
\end{lem}

This will be proved in Section \ref{ssec:pre}. Note that $((Y_1 \cup Y_2) \setminus T) \cup W$ is not necessarily a smoothly embedded submanifold in $X$, but it can be turned into one by smoothing its corners at $B$. So from a smooth embedded submanifold $W \subset T$ with boundary $B$ we can construct a suitable $Y$. 

\begin{thm} \label{thm:top}
We can construct an embedded topological submanifold $W \subset T$ with boundary $B$.
\end{thm}

This will be proved in Section \ref{sec:top}. The construction goes as follows. 

First consider Case 1. Let $F$ be a fiber of $\T$, we can fix an identification between $F$ and $D^2$ (via any local trivialization of $\T$). Then $M \cap F$ is the origin, and $(Y_1 \cup Y_2) \cap F$ corresponds to $D^1 \times \{ 0 \} \cup \{ 0 \} \times D^1$. This can be replaced by two line segments, to get an embedded manifold with the same boundary $S^0 \times \{ 0 \} \cup \{ 0 \} \times S^0$ (corresponding to $B \cap F$), see Figure 1. The union of these two lines is the solution set of the equation 
\[
(x+y-1)(x+y+1)=0 \iff 2xy = 1 - x^2 - y^2 \iff 2xy = 1 - |x|^2 - |y|^2 \tag{$*$}
\] 
We will use the third form of this equation, because it works in Case 2 as well: the solution set of this complex equation is an embedded submanifold of $D^4$ with boundary $S^1 \times \{ 0 \} \cup \{ 0 \} \times S^1$ (see Proposition \ref{prop:v}). 

\begin{rem*}
Note that the other two forms of ($*$) can not work in the complex case. The solution set of the first equation does not contain $S^1 \times \{ 0 \} \cup \{ 0 \} \times S^1$, and the second equation is equivalent to the first one (but not the third one) over $\C$.
\end{rem*}

\begin{figure}[h]
\begin{center}
\includegraphics[scale=0.5]{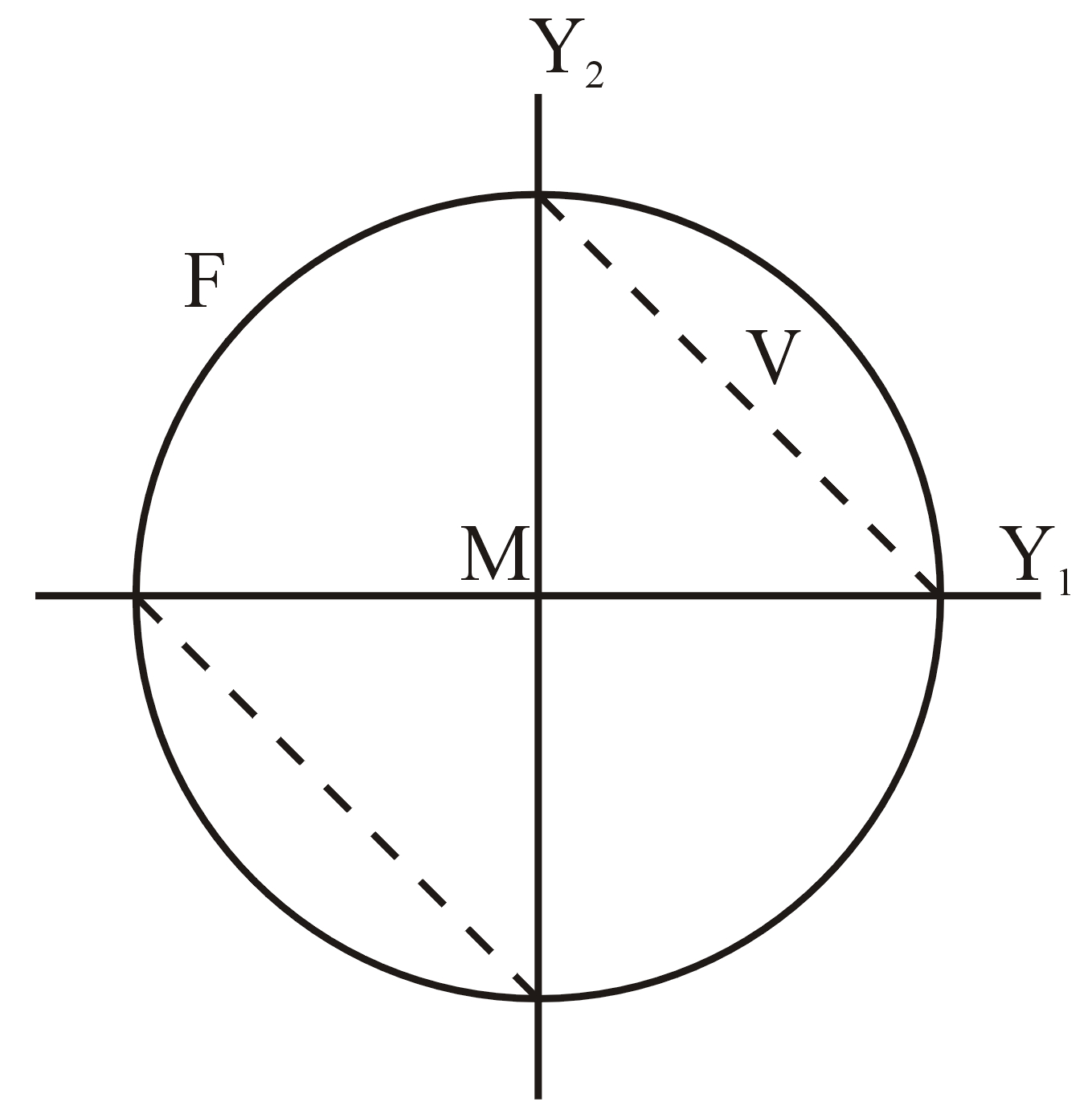}
\end{center}
\begin{center}
Figure 1.
\end{center}
\end{figure}

Let 
\[
V = \left\{ (x,y) \in D^4 \mid 2xy = 1 - |x|^2 - |y|^2 \right\} \text{\,.}
\]

If $T \approx M \times D^4$, ie.\ $\T$ is a trivial bundle, then we can replace $Y_1 \cup Y_2$ with $V$ in every fiber to get $W = M \times V$, which will be a suitable submanfold of $T$. This construction works in a more general situation. There is an action of $U(1)$ on $D^4$ given by $\alpha(x,y) = (\alpha(x), \alpha^{-1}(y))$, and $V$ is invariant under this action. Therefore if the normal bundles of $M$ in $Y_1$ and $Y_2$ are the complex conjugates of each other \big[isomorphic\big], ie.\ $\T$ has a $U(1)$-structure, then $V$ determines a well-defined subset in each fiber of $\T$, and these together form (the total space of) a subbundle $W \subset T$. 

In general, there is a codimension-2 \big[codimension-1\big] submanifold $N$ in $M$ such that $\T$ has a $U(1)$-structure over $M \setminus N$. So we can define the subset $W' \subset T \big| _{M \setminus N}$ as above, it is the union of the subsets $V$ in every fiber. We can also define $W'' \subset T \big| _N$, which contains the subset
\[
\tilde{V} = \left\{ (x,y) \in D^4 \mid 2|xy| \leq 1-|x|^2 - |y|^2 \right\} = \left\{ (x,y) \in D^4 \mid |x| + |y| \leq 1 \right\} 
\]
in every fiber. Note that $\tilde{V}$ is $(U(1) \times U(1))$-invariant, so $W''$ is well-defined. We then define $W = W' \cup W''$. 

We need to check that this $W$ really is a manifold. The proof of this is based on Lemma \ref{lem:loctriv}, which describes the structure of $W'$ around $N$. The statement of Lemma \ref{lem:loctriv} can be interpreted as follows: 

In Case 1 $V$ can be in two possible ``configurations", by which we mean that its image under the action of an element of $O(1) \times O(1)$ is either itself or the subset $\left\{ (x,y) \in D^2 \mid -2xy = 1 - |x|^2 - |y|^2 \right\}$. Any fiber $D^1$ of the normal bundle of the codimension-$1$ submanifold $N \subset M$ over a point $p \in N$ is divided into two parts by $p$ (which corresponds to $0 \in D^1$). If we choose an appropriate local trivialization of $\T$ over this $D^1$ to identify the fibers of $\T$ with $D^2$, then the fibers of $W'$ (that correspond to $V$ when viewed from an appropriate local trivialization of $\T \big| _{M \setminus N}$) will be in one configuration on one side of $p$, and in the other configuration on the other side. This is why we get a manifold when we insert the square $\tilde{V}$ in the fiber over $p$ (see Figure 2).

\begin{figure}[h]
\begin{center}
\includegraphics[scale=0.8]{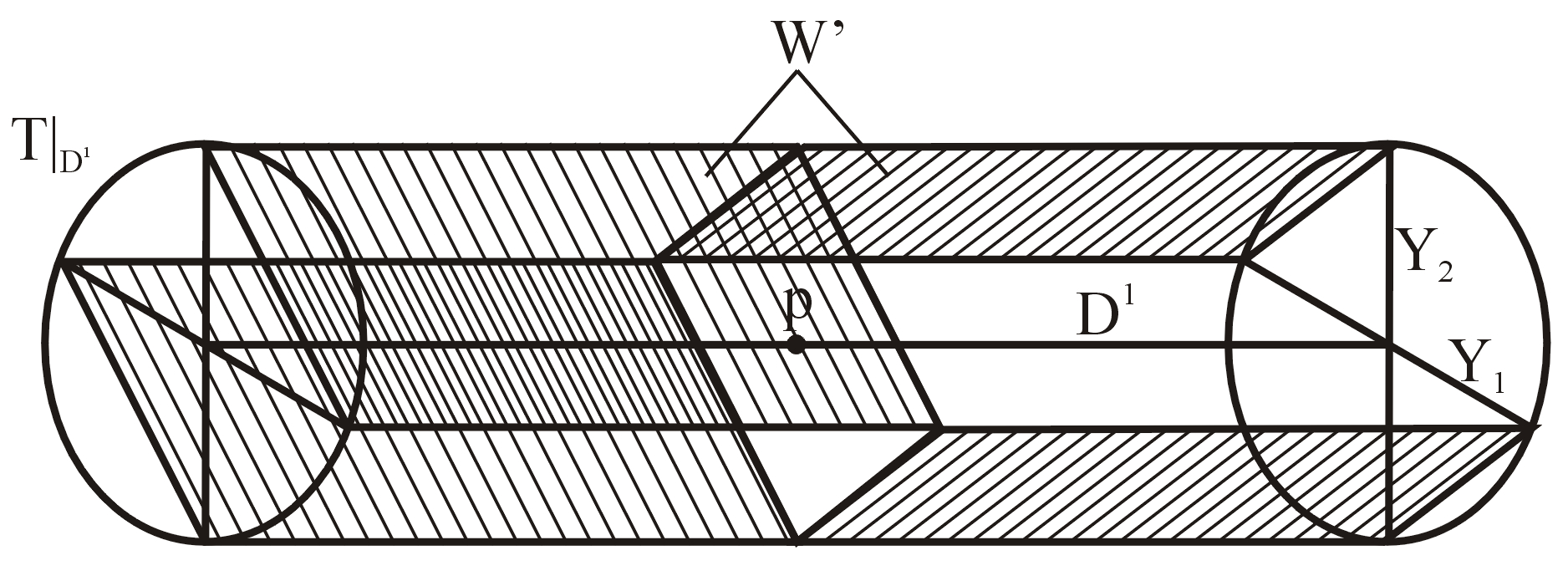}
\end{center}
\begin{center}
Figure 2.
\end{center}
\end{figure}

The situation in Case 2 is analogous. There is an $S^1$-family of possible configurations of $V$, given by $\left\{ (x,y) \in D^4 \mid 2x\theta y = 1 - |x|^2 - |y|^2 \right\}$, where $\theta \in S^1$ is a parameter. A fiber $D^2$ of the normal bundle of the codimension-$2$ submanifold $N \subset M$ can be parametrized by the distance from the origin $t \in [0,1]$ and direction $\theta \in S^1$. We can choose local trivializations (of $\T$, and of the normal bundle of $N \subset M$) such that the fibers of $W'$ over the points whose direction is a given $\theta \in S^1$ will be in the configuration given by $\theta$.

The proof of Lemma \ref{lem:loctriv} is based on the fact that the bundle $\T$ can be pulled back from a universal $(U(1) \times U(1))$-bundle by a map $f : M \rightarrow {\C}P^{\infty} \times {\C}P^{\infty}$ (such that $N = f^{-1}({\C}P^{\infty} \times {\C}P^{\infty-1})$) and on a description of the universal bundle around ${\C}P^{\infty} \times {\C}P^{\infty-1}$.

We will also prove that $W$ has a smooth approximation, more precisely: 

\begin{thm} \label{thm:smooth}
The $W$ constructed above has a smooth structure, and there is a smooth embedding $W \hookrightarrow T$ that is identical on $\partial W = B$, and is homotopic to the identity of $W$ (which can be viewed as a continuous map $W \rightarrow T$). In Case 2 the image $W_2$ of this embedding has a co-orientation that coincides with that of $Y_1 \cup Y_2$ over $B$. 
\end{thm}

In the special case, when $\T$ is a $U(1)$-bundle, the subbundle $W$ is already a smooth submanifold, so we only need to prove that it has an appropriate co-orientation (see Proposition \ref{prop:co-or}).

The general case will be proved in Section \ref{sec:smooth}. We will construct a smooth submanifold $W_2 \subset T$ and a homeomorphism between $W$ and $W_2$. Then the smooth structure of $W_2$ determines a smooth structure of $W$, and the map $W \rightarrow W_2$ will be a smooth embedding. We define $W_2$ by modifying the construction of $W$. Outside of a neighbourhood of $N$ it coincides with $W'$, in the fibers of $\T$ over $N$ it contains the subset $\left\{ (x,y) \in D^4 \mid xy = 0 \right\}$, and in a fiber of $\T$ over a point which is at distance $t \leq 1$ from $N$ we replace $V$ (the fiber of $W'$) with $\left\{ (x,y) \in D^4 \mid 2xy = \ell(t)(1-|x|^2 - |y|^2) \right\}$ for an appropriate $\ell : [0, 1] \rightarrow [0, 1]$. A homeomorphism $W \rightarrow W_2$ will be constructed in Section \ref{ssec:homeo}.

\begin{prop} \label{prop:or1}
(Only in Case 1.) If $Y_1$ and $Y_2$ are both oriented, or both of them are co-oriented, then we can construct a $Y$ with the same property.
\end{prop}

\begin{prop} \label{prop:or2}
If both of $Y_1$ and $Y_2$ are both oriented and co-oriented, then we can construct an oriented and co-oriented $Y$ iff $Y_1$ and $Y_2$ define the same local orientation of $X$ at each point of $M$.
\end{prop}

These propositions will be proved at the end of Section \ref{ssec:sp} and Section \ref{sec:smooth} respectively.

\subsection*{Acknowledgements}

I am grateful to Professor Matthias Kreck for suggesting the problem and to Professor Andr\'as Sz\H{u}cs for his helpful remarks.

\section{The construction of $W$} \label{sec:top}

\subsection{Preliminaries} \label{ssec:pre}

\begin{proof}[Proof of Lemma \ref{lem:w}]
Let $c(W), c((Y_1 \cup Y_2) \cap T) \in C_{n-2}(T; \Z_w)$ and $c((Y_1 \cup Y_2) \setminus T) \in C_{n-2}(X; \Z_w)$ be fundamental chains of the manifolds (with boundary) $W$, etc. For a cycle $c$ let $[c]$ denote its homology class. Let $i : T \hookrightarrow X$ be the inclusion, $i_{\cdot} : C_{n-2}(T; \Z_w) \rightarrow C_{n-2}(X; \Z_w)$, $i_* : H_{n-2}(T; \Z_w) \rightarrow H_{n-2}(X; \Z_w)$. Then $i_*$ is the zero map, because $M$ is a deformation retract of $T$, hence $H_{n-2}(T; \Z_w) = H_{n-2}(M; \Z_w) = 0$. \big[In Case 1 we use $H_{n-1}(T; \Z_2) = H_{n-1}(M; \Z_2) = 0$.\big] Then we have
\[
\begin{aligned}
\ [Y] &= [c((Y_1 \cup Y_2) \setminus T) + i_{\cdot}(c(W))] \\
&= [c((Y_1 \cup Y_2) \setminus T) + i_{\cdot}(c((Y_1 \cup Y_2) \cap T)) + \break i_{\cdot}(c(W) - c((Y_1 \cup Y_2) \cap T))] \\
&= [c((Y_1 \cup Y_2) \setminus T) + i_{\cdot}(c((Y_1 \cup Y_2) \cap T))] + i_*([c(W) - \break c((Y_1 \cup Y_2) \cap T)]) \\
&= [Y_1 \cup Y_2] + 0 = [Y_1] + [Y_2] \text{\,.}
\end{aligned}
\]
\end{proof}

We have claimed that the structure group of $\T$ is $U(1) \times U(1)$, now we shall make this statement more precise. The structure group of $\T$ is, in fact, a certain subgroup of $\Homeo(D^4)$. It is isomorphic to $U(1) \times U(1)$ via the isomorphism specified by the action $(\alpha, \beta)(x,y)=(\alpha(x), \beta(y))$ (recall that an action of a group $G$ on a space $F$ is just a homomorphism $G \rightarrow \Homeo(F)$). 

But we can define another isomorphism between $U(1) \times U(1)$ and the structure group of $\T$, it is given by the action $(\alpha, \beta)(x,y)=(\alpha(x), \alpha^{-1}\beta(y))$ (this is an isomorphism, because it is the composition of the previous action with the automorphism of $U(1) \times U(1)$, $(\alpha, \beta) \mapsto (\alpha, \alpha^{-1}\beta)$). The reason for using this isomorphism instead of the first one is that we want to apply the following lemma to it.

\begin{lem} \label{lem:gr}
Let $F$ be a space, and  $G_1, G_2 \leq G_{12} \leq \Homeo(F)$ such that $G_{12} = G_1 \times G_2$. Then $BG_{12} = BG_1 \times BG_2$ (up to homotopy equivalence). There exist universal bundles $p_i : E_i \rightarrow BG_i$ with fiber $F$ and structure group $G_i$ for $i = 1, 2, 12$. Since $G_1 \leq G_{12}$, $p_1$ is also a $G_{12}$-bundle. We claim that the inclusion $BG_1 \approx BG_1 \times \ast \subseteq BG_1 \times BG_2$ induces $p_1$ from $p_{12}$. (And an analogous statement holds for $G_2$.)
\end{lem}

\begin{proof}
Let $\bar{p}_i : EG_i \rightarrow BG_i$ be the universal principal $G_i$-bundle, then $EG_{12} = EG_1 \times EG_2$, and $E_i = EG_i \times_{G_i} F$. It is easy to check that the restriction of the bundle $\bar{p}_{12}$ to $BG_1 \times \ast$ is the composition $EG_1 \times G_2 \rightarrow EG_1 \rightarrow BG_1$. And $E_1 = EG_1 \times_{G_1} F = (EG_1 \times G_2) \times_{G_1 \times G_2} F$, so $p_1$ is the restriction of $p_{12}$ to $BG_1 \times \ast$.
\end{proof}

Let $\xi$ be the universal bundle with fiber $D^4$ and structure group $U(1) \times U(1)$ (identified with a subgroup of $\Homeo(D^4)$ via the action $(\alpha, \beta)(x,y)=(\alpha(x), \alpha^{-1}\beta(y))$). Its base space is $BU(1)\times BU(1) = {\C}P^{\infty} \times {\C}P^{\infty}$. The bundle $\T$ can be induced from $\xi$ by a homotopically unique map $f : M \rightarrow {\C}P^{\infty} \times {\C}P^{\infty}$.

\subsection{Special case} \label{ssec:sp}

First consider the special case when $f$ goes into ${\C}P^{\infty} \times \ast$. By Lemma \ref{lem:gr} $\xi \big| _{{\C}P^{\infty} \times \ast}$ is the universal bundle with fiber $D^4$, structure group $U(1)$ and action $\alpha(x,y)=(\alpha(x), \alpha^{-1}(y))$. So in this case $\T$ also has a bundle structure of this type. Note that $\T$ has a $U(1)$-structure ($f$ is homotopic to a map into ${\C}P^{\infty} \times \ast$) iff the normal bundles of $M$ in $Y_1$ and $Y_2$ are the complex conjugates of each other \big[isomorphic\big]. 

Let $F$ be a fiber of $\T$, and choose an identification of $F$ with $D^4$ (coming from a local trivialization of $\T$, so this identification can be changed by the action of $U(1)$). Let 
\[
V = \left\{ (x,y) \in F \approx D^4 \subset \C^2 \mid 2xy = 1 - |x|^2 - |y|^2 \right\}.
\]

This $V$ is invariant under the action of $U(1)$, hence it is well-defined (independent of the choice of the identification of $F$ with $D^4$). Note that if we define $V$ in all fibers of $\T$, they will form a locally trivial bundle $W$ over $M$.

\begin{prop} \label{prop:v}
The subset $V$ defined above is a smooth submanifold of $F$.
\end{prop}

\begin{proof}[Proof in Case 2]
Consider the function $v : \R^4 \rightarrow \R^2$, 
\[
v(x_1,x_2,y_1,y_2) = (2x_1y_1-2x_2y_2-1+x_1^2+x_2^2+y_1^2+y_2^2 \, ,\: 2x_1y_2+2x_2y_1) \text{\,.}
\] 
Then $V = D^4 \cap v^{-1}(0,0)$ (if we identify $F$ with $D^4$), hence it is enough to prove that $(0,0)$ is a regular value of $v$. So we need to find the critical points of $v$.
\[
\mathrm{d}v = 2
\begin{bmatrix} 
y_1+x_1 & -y_2+x_2 & x_1+y_1 & -x_2+y_2 \\
y_2 & y_1 & x_2 & x_1
\end{bmatrix}
\]
$(x_1,x_2,y_1,y_2)$ is critical iff $\rk \mathrm{d}v <2$ at that point, that is, all $2 \times 2$ minors of the matrix vanish:
\begin{alignat}{4}
y_1^2 &\:+&\: x_1y_1 &\:-&\: y_2^2 &\:+&\: x_2y_2 &\:= 0 \\
y_1x_2 &\:+&\: x_1x_2 &\:+&\: x_1y_2 &\:+&\: y_1y_2 &\:= 0 \\
y_1x_1 &\:+&\: x_1^2 &\:-&\: x_2y_2 &\:+&\: y_2^2 &\:= 0 \\
-y_2x_2 &\:+&\: x_2^2 &\:+&\: x_1y_1 &\:+&\: y_1^2 &\:= 0 \\
-y_2x_1 &\:+&\: x_2x_1 &\:-&\: x_2y_1 &\:+&\: y_2y_1 &\:= 0 \\
x_1^2 &\:+&\: y_1x_1 &\:-&\: x_2^2 &\:+&\: y_2x_2 &\:= 0 
\end{alignat}
Substracting (4) from (3) we get $x_1^2 + y_2^2 - x_2^2 - y_1^2 = 0$, hence
\[
x_1^2 - y_1^2 = x_2^2 - y_2^2 \tag{7} 
\]
(2) means that
\[
(x_1 + y_1)(x_2 + y_2) = 0
\]
It follows that either $|x_1| = |y_1|$ or $|x_2| = |y_2|$. By (7) one of these implies the other, so we know that $|x_1| = |y_1|$ \textit{and} $|x_2| = |y_2|$. So there exist $\varepsilon_1 = \pm 1$ and $\varepsilon_2 = \pm 1$ such that $y_1 = \varepsilon_1x_1$ and $y_2 = \varepsilon_2x_2$. By (3) $\varepsilon_1x_1^2 + x_1^2 - \varepsilon_2x_2^2 + x_2^2 = 0$, hence
\[
(\varepsilon_1+1)x_1^2 + (-\varepsilon_2+1)x_2^2 = 0
\]
Both of the summands are non-negative, so both of them must be $0$. It follows that $(\varepsilon_1+1)x_1 = 0$ and $(-\varepsilon_2+1)x_2 = 0$, that is $y_1 + x_1 = 0$ and $-y_2 + x_2 = 0$. But in this case $v(x_1,x_2,y_1,y_2) = (-1,0)$, so this is the only critical value of $v$. Therefore $(0,0)$ is a regular value, and $V$ is a submanifold.
\end{proof}

\begin{proof}[Proof in Case 1] 
$2xy = 1 - x^2 - y^2 \Leftrightarrow (x+y-1)(x+y+1)=0$, so $V$ is the union of two parallel lines (intersected by $D^2$).
\end{proof}

\begin{prop}
The boundary of $V$ is $ B \cap F$.
\end{prop}
\begin{proof}
\[
\begin{aligned}
\quad V \cap \partial F &= \left\{ (x,y) \in D^4 \mid 2xy = 1-|x|^2 - |y|^2, |x|^2 + |y|^2 = 1 \right\} = \\
 &= \left\{ (x,y) \in D^4 \mid xy = 0, |x|^2 + |y|^2 = 1 \right\} = \\
 &= \left\{ (x,y) \in D^4 \mid (y = 0 \text{ and } |x| = 1) \text{ or } (x = 0 \text{ and } |y| = 1) \right\} = \\
 &= S^1 \times \{ 0 \} \bigsqcup \{ 0 \} \times S^1 = B \cap F \text{\,.}
\end{aligned}
\]
It is easy to check that $V$ is transversal to $\partial F$ (for a point $p \in V \cap \partial F$ the restriction of $\mathrm{d}v(p)$ to the tangent space $\Tang_p(\partial F)$ of $\partial F$ is a rank-2 linear map). Therefore $\partial V = V \cap \partial F = B \cap F$.
\end{proof}

\begin{prop}
$V \approx D^1 \times S^1$. \quad \big[$V \approx D^1 \times S^0$.\big]
\end{prop}

\begin{proof}
Let $\pi : V \rightarrow D^2$ be the projection $(x,y) \mapsto x$. Then $\pi^{-1}(0) = \{ 0 \} \times S^1$, and $\pi^{-1}(x) = (x,0)$ if $|x| = 1$. If $0 < |x| < 1$, and $(x,y) \in \pi^{-1}(x)$, then $(x,y)$ is an interior point of $V$, and of $D^4$. Therefore $1 - |x|^2 - |y|^2 > 0$, so from $2xy = 1 - |x|^2 - |y|^2$ it follows that $y \neq 0$, and $\arg y = - \arg x$. Also $2|xy| = 1 - |x|^2 - |y|^2 \Leftrightarrow (|y| + |x|)^2 = 1 \Leftrightarrow |y| = 1 - |x|$ . Therefore $\pi^{-1}(x)$ consists of a single point for all $x$ with $0 < |x| < 1$, so $\interior V \approx \interior D^2 \setminus \{ 0 \}$.
\end{proof}

Recall that $W$ is the union of $V$s in all fibers of $\T$. By the above propositions this is a (smooth) embedded submanifold in $T$ with boundary $B$. So in Case 1 $Y = ((Y_1 \cup Y_2) \setminus T) \cup W$ is a suitable representative. In Case 2 we need to prove that it is co-oriented.

\begin{prop} \label{prop:co-or}
$Y$ has a co-orientation which extends those of $Y_1 \setminus T$ and $Y_2 \setminus T$. 
\end{prop}

\begin{proof}
We have to show that $W$ is co-orientable, and its co-orientation coincides with that of $(Y_1 \cup Y_2) \setminus T$ over $B$. The latter is done by identifying the restrictions to $B$ of the two normal bundles, and showing that this identification preserves orientation.

The normal bundle of $W$ in $X$, restricted to $V = F \cap W$ in any fiber $F$ of $\T$ is just the normal bundle of $V$ in $F$. For any point $p \in V$, the rows of $\mathrm{d}v(p)$ (elements of $\R^4 = \Tang_pF$, the tangent space of $F$ at $p$) form a basis of the normal space of $V$ (they are, by definition, orthogonal to $\Ker \mathrm{d}v(p) = \Tang_pV$, and they are linearly independent). This basis, defined at each point $p$, determines a co-orientation of $V$. (This basis is independent of the choice of the identification of $F$ with $D^4$, hence the co-orientation is well-defined. Indeed, in the point $p=(x,y)$ the vectors are $(x+\bar{y},y+\bar{x})$ and $(i\bar{y},i\bar{x})$, and these are equivariant under the $U(1)$-action.)

The normal bundle of $Y_1$ in $X$, restricted to $D^2 \times \{ 0 \} = F \cap Y_1$ is the normal bundle of $D^2 \times \{ 0 \}$ in $F$, in any fiber of it the vectors $(0,0,1,0)$, $(0,0,0,1)$ form a positively oriented basis (if an identification $F \approx D^4$ is fixed).

At a point $p = (x_1,x_2,0,0) \in S^1 \times \{ 0 \}$ the projections of the rows of $\mathrm{d}v(p)$ to the sub\-space $\left<(0,0,1,0),(0,0,0,1)\right>$ are $(0,0,x_1,-x_2)$ and $(0,0,x_2,x_1)$. Since $\det \bigl( \begin{smallmatrix} x_1 & -x_2 \\ x_2 & \hfill x_1 \end{smallmatrix} \bigr) = x_1^2 + x_2^2 = 1 > 0$, we know that the projection is an isomorphism, so it identifies the normal spaces of $V$ and $D^2 \times \{ 0 \}$, and it preserves orientation.

We can show analogously that the co-orientations of $W$ and $Y_2$ coincide over $\{ 0 \} \times S^1$.
\end{proof}

\begin{proof}[Proof of Proposition \ref{prop:or1}]
First note that if $Y_1$ and $Y_2$ are both co-oriented, then their normal bundles are trivial, so $\T$ is a trivial bundle too. 

If $Y_1$ and $Y_2$ are oriented, then $\T$ is an orientable $D^2$-bundle. An orientation can be constructed as follows. At any point of $M$ choose a local orientation of $M$. Choose a normal vector of $M$ in $Y_1$ such that the local orientation and this vector determine the positive orientation of $Y_1$. Choose a normal vector in $Y_2$ in a similar way. Then these two vectors determine an orientation of the fiber of $\T$ at this point, and this orientation is well-defined (independent of the choice of the local orientation of $M$).

So in both cases we are in the special case ($\T$ can be induced from ${\R}P^{\infty} \times \ast$).

We will define $W$ in a slightly different way than before. Let $F$ be a fiber of $\T$, and
\[
\begin{aligned}
V_1 &= \left\{ (x,y) \in F \approx D^2 \subset \R^2 \mid 2xy = 1 - x^2 - y^2 \right\} \\
V_2 &= \left\{ (x,y) \in F \approx D^2 \subset \R^2 \mid -2xy = 1 - x^2 - y^2 \right\} 
\end{aligned}
\]
These are well-defined (invariant under the $O(1)$-action). In fact $V_1$ is just $V$, and $V_2$ is its image under the isometry $t : D^2 \rightarrow D^2$, $t(x,y) = (x,-y)$. Therefore $V_2$ is a submanifold too, and is also diffeomorphic to $D^1 \times S^0$. The sets $S^0 \times \{ 0 \}$ and $\{ 0 \} \times S^0$ are invariant under $t$, so $\partial V_1 = \partial V_2$.

In any fiber $F$ of $\T$ we will choose one of $V_1$ and $V_2$ according to the following rules.

If $Y_1$ and $Y_2$ are co-oriented, then their normal bundles, restricted to the subsets $Y_1 \cap F = D^1 \times \{ 0 \}$ and $Y_2 \cap F = \{ 0 \} \times D^1$ are just the normal bundles of $D^1 \times \{ 0 \}$ and $\{ 0 \} \times D^1$ in $F$. Exactly one of $V_1$ and $V_2$ has a co-orientation in $F$ which is compatible with the orientations of these normal bundles, this $V_i$ will be chosen. (See Figure 3.)

If $Y_1$ and $Y_2$ are oriented, then in the fiber $F$ of $\T$ over $p \in M$ we choose $V_i$ the following way. First we choose a local orientation of $M$ at $p$. We orient $D^1 \times \{ 0 \}$ such that the local orientation of $M$, and this orientation together determine the positive orientation of $Y_1$ at $p$. We orient $\{ 0 \} \times D^1$ in a similar way. Exactly one of $V_1$ and $V_2$ has an orientation which is compatible with these orientations, this $V_i$ will be chosen. (See Figure 4.) (This choice is independent of the choice of local orientation of $M$.)

\begin{figure}[h]
\begin{center}
\includegraphics[scale=0.5]{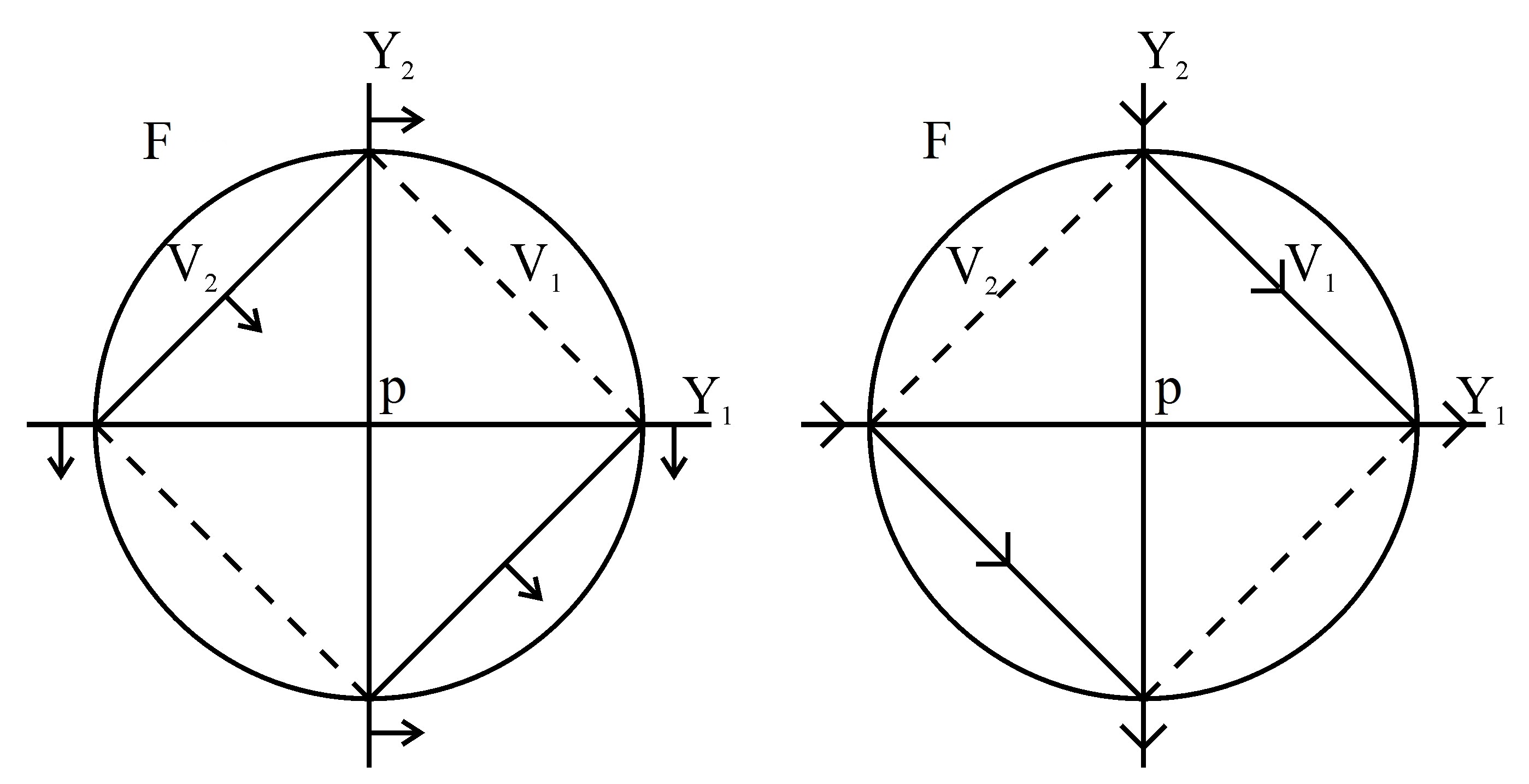}
\end{center}
\begin{center}
Figures 3--4.
\end{center}
\end{figure}

Let $W$ be the subset of $T$ which is the union of the chosen $V_i$-s in all fibers. Then $W$ is a locally trivial bundle over $M$ with fiber $D^1 \times S^0$, and its boundary is $\partial W = B$. By construction $W$ has a co-orientation or orientation compatible with that of $Y_1$ and $Y_2$.
\end{proof}

\subsection{General case} \label{ssec:gen}

First let us recall a few facts about ${\C}P^{\infty}$.

Let $\C^{\infty} = \{ a = (a_0, a_1, a_2, \dots) \mid a_i \in \C \text{, $a_i \neq 0$ for only finitely many $i$} \}$, then \break ${\C}P^{\infty} = (\C^{\infty} \setminus \{ 0 \}) / {\sim}$, where $a \sim \lambda a$ for all $a \in \C^{\infty} \setminus \{ 0 \}$ and $\lambda \in \C \setminus \{ 0 \}$. We define $\C^{\infty - 1} = \{ a \in \C^{\infty} \mid a_0 = 0 \}$, and ${\C}P^{\infty - 1} = (\C^{\infty - 1} \setminus \{ 0 \}) / {\sim}$. 

For $a \in \C^{\infty}$ let $a' = (a_1, a_2, \dots)$, so $a = (a_0, a')$, and let $[a]$ be its image in ${\C}P^{\infty}$, if $a \neq 0$. The norm $\| a \|$ makes sense in $\C^{\infty}$ too, since the sum $\| a \|^2 = |a_0|^2 + |a_1|^2 + \dots$ is finite. 

Let $\ast = [1, 0, 0, \dots] \in {\C}P^{\infty}$, then ${\C}P^{\infty-1}$ is a deformation retract of ${\C}P^{\infty} \setminus \ast$. In fact we can define a bundle structure with fiber $\C$ and structure group $U(1)$ on ${\C}P^{\infty} \setminus \ast$, which turns it into the normal bundle $\nu$ of ${\C}P^{\infty-1}$ in ${\C}P^{\infty}$. We choose a system of contractible open subsets $\hat{U}_i$ covering ${\C}P^{\infty-1}$. Over such a $\hat{U}_i$ we can continuously choose representatives $(0,b')$ for each point $[0,b']$ such that $\| b' \| = 1$ (this is equivalent to choosing a section of the tautological $S^1$-bundle over ${\C}P^{\infty-1}$, which can be done over a contractible subset). Consider the maps $\hat{\varphi}_i : \hat{U}_i \times \C \rightarrow {\C}P^{\infty} \setminus \ast$, \linebreak $([0, b'], g) \mapsto [g, b']$. These define a locally trivial bundle, and the structure group is $U(1)$, because if $[0, b'_1] = [0, b'_2] \in \hat{U}_{i_1} \cap \hat{U}_{i_2}$, then $\hat{\varphi}_{i_1}([0, b'_1], g_1) = \hat{\varphi}_{i_2}([0, b'_2], g_2)$ iff $g_2 = \lambda g_1$, where $\lambda$ is determined by $b'_2 = \lambda b'_1$, so $\lambda \in U(1)$ (since $\| b'_1 \| = \| b'_2 \| = 1$). Let $\hat{\pi} : {\C}P^{\infty} \setminus \ast \rightarrow {\C}P^{\infty-1}$ denote the projection of $\nu$.

Similarly let $\tilde{\nu}$ be the normal bundle of ${\C}P^{\infty} \times {\C}P^{\infty-1}$ in ${\C}P^{\infty} \times {\C}P^{\infty}$, which is the pullback of $\nu$ by the projection ${\C}P^{\infty} \times {\C}P^{\infty-1} \rightarrow {\C}P^{\infty-1}$. Denote the projection of $\tilde{\nu}$ by $\tilde{\pi} : {\C}P^{\infty} \times ({\C}P^{\infty} \setminus \ast) \rightarrow {\C}P^{\infty} \times {\C}P^{\infty-1}$. 

The space ${\C}P^{\infty} \setminus {\C}P^{\infty-1}$ is contractible ($\ast$ is its deformation retract, a homotopy between the identity and the constant map is given by $h_t([1, c']) = [1, tc']$, $t \in [0,1]$).

\bek
Now let us return to our original problem. We consider the case of a general $f : M \rightarrow {\C}P^{\infty} \times {\C}P^{\infty}$. 

We may assume that $f$ is transverse to ${\C}P^{\infty} \times {\C}P^{\infty-1}$. Then we can define $N = f^{-1}({\C}P^{\infty} \times {\C}P^{\infty-1})$, this is a codimension-2 co-orientable \big[codimension-1 arbitrary\big] submanifold in $M$. 

The bundles $\xi \big| _{{\C}P^{\infty} \times ({\C}P^{\infty} \setminus {\C}P^{\infty-1})}$ and $(\id \times h_0)^*\big(\xi \big| _{{\C}P^{\infty} \times \ast}\big)$ are isomorphic (because $\id \times h_0$ is a deformation retraction). The latter is a $U(1)$-bundle, so this is also true for the former. Since $\T$ is induced from $\xi$, over $f^{-1}({\C}P^{\infty} \times ({\C}P^{\infty} \setminus {\C}P^{\infty-1})) = M \setminus N$ the restriction of $\T$ is also an $U(1)$-bundle, as in the previous case. So we can apply the previous construction, we denote the result by $W' \subset T \big| _{M \setminus N}$. (Note that this depends on the isomorphism between $\xi \big| _{{\C}P^{\infty} \times ({\C}P^{\infty} \setminus {\C}P^{\infty-1})}$ and $(\id \times h_0)^*\big(\xi \big| _{{\C}P^{\infty} \times \ast}\big)$, we will specify this isomorphism later.)

Let $S$ be a tubular neighbourhood of $N$ in $M$, it is the total space of a bundle $\S$ with fiber $D^2$ and structure group $U(1)$. This is the pullback of the disk bundle of $\tilde{\nu}$ by $f \big| _N : N \rightarrow {\C}P^{\infty} \times {\C}P^{\infty-1}$. Let $\pi : S \rightarrow N$ denote the projection of $\S$. The following lemma describes the structure of $W'$ around $N$.

\begin{lem} \label{lem:loctriv}
There exist 
\begin{compactitem}
\item trivializing neighbourhoods $U_k \subseteq N$ for both $\S$ and $\T \big| _{N}$, 
\item local trivializations $\varphi_k : U_k \times D^2 \rightarrow S \big| _{U_k}$, $\psi_k : U_k \times D^4 \rightarrow T \big| _{U_k}$, and 
\item an isomorphism $H : T \big| _{S} \rightarrow \pi^*(T \big| _{N})$,
\end{compactitem}
such that 
\begin{compactitem}
\item if $\varphi_{k_1,k_2} : U_{k_1} \cap U_{k_2} \rightarrow U(1)$ denotes the transition map between $\varphi_{k_1}$ and $\varphi_{k_2}$ (ie.\ at each point $p$ its value is the map $\varphi_{k_2}^{-1} \circ \varphi_{k_1} \big| _{\{p\} \times \C} : \C \rightarrow \C$, expressed as an element of $U(1)$), $\overline{\varphi_{k_1,k_2}}$ is its complex conjugate, 

$\psi_{k_1,k_2} : U_{k_1} \cap U_{k_2} \rightarrow U(1) \times U(1)$ is the transition map between $\psi_{k_1}$ and $\psi_{k_2}$, 

and $\pr_2 : U(1) \times U(1) \rightarrow U(1)$ is the projection to the second component,

then $\overline{\varphi_{k_1,k_2}} = \pr_2 \circ \psi_{k_1,k_2}$; and
\item if $\psi'_k : S \big| _{U_k} \times D^4 \rightarrow T \big| _{S|_{U_k}}$ denotes the trivialization coming from $\psi_k$ and $H$, ie.\ $\psi'_k(q,(x,y)) = H^{-1}(q, \psi_k(p,(x,y)))$ (where $q \in S \big| _{U_k}$, $p = \pi(q) \in U_k$, $(x,y) \in D^4$, and $(q, \psi_k(p,(x,y))) \in \pi^*(T \big| _{N})$ -- recall that $\pi^*(T \big| _{N})$ is defined via a pullback diagram, so it is a subset of $S \times T \big| _{N}$),

then for any $q = \varphi_k(p, t \theta) \in S \big| _{U_k}$ ($t \in (0,1]$, $\theta \in S^1$) if we use $\psi'_k$ to identify the fiber $F_q$ of $\T$ over $q$ and $D^4$, then  
\[
F_q \cap W' = \left\{ (x,y) \in F_q \approx D^4 \mid 2x\theta y = 1-|x|^2 - |y|^2 \right\} \text{\,.}
\]
\end{compactitem}
\end{lem}

\begin{proof}
Let $\gamma$ denote the tautological line bundle over ${\C}P^{\infty}$. 

Then $\gamma \times \gamma$ (over ${\C}P^{\infty} \times {\C}P^{\infty}$) is the universal bundle with fiber $\C^2$ and structure group $U(1) \times U(1)$ (with its usual action on $\C^2$). This is identical to the universal bundle $\xi$ in the sense that a system of trivializing neighbourhoods and abstract transition maps for one of them will form such a system for the other one as well. (By abstract transition maps we mean maps from intersections of trivializing neighbourhoods to $U(1) \times U(1)$. To produce actual gluing maps we need to compose these with the appropriate action of $U(1) \times U(1)$ on the fiber.) Since $\T$ is induced from $\xi$, and local trivializations can be pulled back, we can construct local trivializations of $\T$ via those of $\gamma \times \gamma$.

We can choose contractible open subsets $\hat{U}'_j$ covering ${\C}P^{\infty}$. If in each $\hat{U}'_j$ we fix a representative $a$ for each point $[a]$, these determine a map $\hat{\sigma}_j : \hat{U}'_j \times \C \rightarrow \gamma \big| _{\hat{U}'_j}$, $([a], x) \mapsto ([a], xa)$ which is a trivialization of $\gamma \big| _{\hat{U}'_j}$.

Consider all possible products $\hat{U}'_j \times \hat{U}_i$. These will form a system of contractible open subsets $\tilde{U}_k$ covering ${\C}P^{\infty} \times {\C}P^{\infty-1}$. The map $\tilde{\varphi}_k = \id_{\hat{U}'_j} \times \hat{\varphi}_i : \tilde{U}_k \times \C \rightarrow \break {\C}P^{\infty} \times ({\C}P^{\infty} \setminus \ast)$, $(([a], [0, b']), g) \mapsto ([a], [g, b'])$ is a local trivialization of $\tilde{\nu}$. 

Let $\hat{\psi}_i : \hat{U}_i \times \C \rightarrow \gamma \big| _{\hat{U}_i}$, $([0, b'], y) \mapsto ([0, b'], (0, y b'))$ be a local trivialization of \allowbreak $\gamma \big| _{{\C}P^{\infty-1}}$. For each $\tilde{U}_k = \hat{U}'_j \times \hat{U}_i$ we define the map $\tilde{\psi}_k : \tilde{U}_k \times \C^2 \rightarrow (\gamma \times \gamma) \big| _{\tilde{U}_k}$ by $(([a], [0, b']), \allowbreak (x, y)) \mapsto \big(([a], [0, b']), (xa, (0, y b')) \big)$. (This is in fact $\hat{\sigma}_j \times \hat{\psi}_i$.)

Let $\hat{H} : \gamma \big| _{{\C}P^{\infty} \setminus \ast} \rightarrow \hat{\pi}^*\left(\gamma \big| _{{\C}P^{\infty-1}}\right)$, $([d], \lambda d) \mapsto ([d], ([0, d'], (0, \lambda \frac{\| d \|}{\| d' \|} d')))$, this is an isomorphism (it is well-defined, ie.\ independent of the representative $d$, and depends only on $[d]$ and $\lambda d$). If we multiply it by the identity of $\gamma$, we get an isomorphism $\tilde{H} : (\gamma \times \gamma) \big| _{{\C}P^{\infty} \times ({\C}P^{\infty} \setminus \ast)} \rightarrow \tilde{\pi}^*\left((\gamma \times \gamma) \big| _{{\C}P^{\infty} \times {\C}P^{\infty-1}}\right)$.

Let $U_k = f^{-1}(\tilde{U}_k)$. Since $\S$ is the pullback of $\tilde{\nu}$ by $f$, and $\tilde{\varphi}_k$ is a local trivialization of $\tilde{\nu}$ over $\tilde{U}_k$, it can also be pulled back to a local trivialization $\varphi_k$ of $\S$ over $U_k$. Since $\tilde{\psi}_k$ is a local trivialization of $\gamma \times \gamma$ over $\tilde{U}_k$, it determines a local trivialization of $\xi$ over $\tilde{U}_k$, which in turn can be pulled back to a local trivialization $\psi_k$ of $\T$ over $U_k$. Finally $\tilde{H}$ determines an isomorphism between $\xi \big| _{{\C}P^{\infty} \times ({\C}P^{\infty} \setminus \ast)}$ and $\tilde{\pi}^*\left(\xi \big| _{{\C}P^{\infty} \times {\C}P^{\infty-1}}\right)$, which is pulled back to an isomorphism $H$ between $T \big| _{S}$ and $\pi^*(T \big| _{N})$.

We need to prove that these maps satisfy the conditions of the Lemma.

\bek
For the first condition it is enough to check that $\overline{\tilde{\varphi}_{k_1,k_2}} = \pr_2 \circ \tilde{\psi}_{k_1,k_2}$ (these are the transition maps between $\tilde{\varphi}_{k_1}$ and $\tilde{\varphi}_{k_2}$, and between $\tilde{\psi}_{k_1}$ and $\tilde{\psi}_{k_2}$), because transition maps between local trivializations of $\gamma \times \gamma$ determine those of $\xi$, and are pulled back to those of $\T$.

So assume that $\tilde{U}_{k_1} = \hat{U}'_{j_1} \times \hat{U}_{i_1}$ and $\tilde{U}_{k_2} = \hat{U}'_{j_2} \times \hat{U}_{i_2}$ have non-empty intersection. Let $([a_1],[0, b_1]) = ([a_2],[0, b_2])$ be an element of this intersection (where the representative $a_1$ of $[a_1]$ is the one chosen in $\hat{U}'_{j_1}$, etc.). Then the value of $\tilde{\varphi}_{k_1,k_2}$ in this point is $\lambda$ where $b_2 = \lambda b_1$. The value of $\tilde{\psi}_{k_1,k_2}$ is $(\mu, \lambda^{-1})$, where  $a_1 = \mu a_2$ (and  $b_1 = \lambda^{-1} b_2$). Since $\pr_2(\mu, \lambda^{-1}) = \lambda^{-1} = \overline{\lambda}$, this proves that the first condition is satisfied.

\begin{rem*}
Similarly we can show that $\overline{\hat{\varphi}_{i_1,i_2}} = \hat{\psi}_{i_1,i_2}$, which proves the well-known fact that the normal bundle $\nu$ of ${\C}P^{\infty-1}$ in ${\C}P^{\infty}$ is the complex conjugate of the tautological bundle $\gamma \big| _{{\C}P^{\infty-1}}$ \big[or in the real case $\nu \cong \gamma \big| _{{\R}P^{\infty-1}}$\big]. 
\end{rem*}

The local trivialization $\tilde{\psi}_k$ and the isomorphism $\tilde{H}$ determine a local trivialization $\tilde{\psi}'_k : \tilde{\pi}^{-1}(\tilde{U}_k) \times \C^2 \rightarrow (\gamma \times \gamma) \big| _{\tilde{\pi}^{-1}(\tilde{U}_k)}$ as follows. Let $([a], [b_0, b']) \in \tilde{\pi}^{-1}(\tilde{U}_k)$ (then $([a], [0, b']) \in \tilde{U}_k = \hat{U}'_j \times \hat{U}_i$, and we assume that $a$ and $(0,b')$ are the previously chosen representatives, so $\| b' \| = 1$) and $(x,y) \in \C^2$. We first apply $\tilde{\psi}_k$ to $(([a], [0,b']), (x, y))$ to get $\big(([a], [0,b']), (x a, (0, y b'))\big) \in (\gamma \times \gamma) \big| _{\tilde{U}_k}$. This determines an element $\big(\big([a],[b_0, b']\big),\big(([a], [0,b']), (x a, (0, y b'))\big)\big)$ in the fiber of $\tilde{\pi}^*\left((\gamma \times \gamma) \big| _{\tilde{U}_k}\right)$ over $([a],[b_0, b'])$. Finally we apply $\tilde{H}^{-1}$ to get $\big(([a],[b_0, b'])$, $(x a, (\frac{y}{\| b \|} b_0, \frac{y}{\| b \|} b'))\big)$. To sum up, $\tilde{\psi}'_k(([a], [b]), (x, y)) = (([a],[b]),(x a, \frac{y}{\| b \|} b))$.

The local trivilaization $\psi'_k$ is constructed in the same way from $\psi_k$ and $H$ as $\tilde{\psi}'_k$ from $\tilde{\psi}_k$ and $\tilde{H}$, therefore it is the pullback of $\tilde{\psi}'_k$.

Let $\omega : ({\C}P^{\infty} \setminus {\C}P^{\infty-1}) \times \C \rightarrow \gamma \big| _{{\C}P^{\infty} \setminus {\C}P^{\infty-1}}$ be the map defined by the formula $([1, c'], y) \mapsto ([1, c'], (\frac{y}{\| (1, c') \|}$, $\frac{y}{\| (1, c') \|} c') )$, it is a trivialization of $\gamma \big| _{{\C}P^{\infty} \setminus {\C}P^{\infty-1}}$. This is an isomorphism between the trivial bundle $h_0^*(\gamma \big| _{\ast})$ and $\gamma \big| _{{\C}P^{\infty} \setminus {\C}P^{\infty-1}}$, and it induces an isomorphism  between $(\gamma \times \gamma) \big| _{{\C}P^{\infty} \times ({\C}P^{\infty} \setminus {\C}P^{\infty-1})}$ and $(\id \times h_0)^*\big((\gamma \times \gamma) \big| _{{\C}P^{\infty} \times \ast}\big)$. It determines an isomorphism between $\xi \big| _{{\C}P^{\infty} \times ({\C}P^{\infty} \setminus {\C}P^{\infty-1})}$ and $(\id \times h_0)^*\big(\xi \big| _{{\C}P^{\infty} \times \ast}\big)$, we assume that this is used to define $W'$.

Let $\tilde{\sigma}_j : \hat{U}'_j \times ({\C}P^{\infty} \setminus {\C}P^{\infty-1}) \times \C^2 \rightarrow (\gamma \times \gamma) \big| _{\hat{U}'_j \times ({\C}P^{\infty} \setminus {\C}P^{\infty-1})}$ be the map given by $([a], [1, c'], x, y) \mapsto ([a], [1, c'], x a, (\frac{y}{\| (1, c') \|}, \frac{y}{\| (1, c') \|} c'))$. (This is $\hat{\sigma}_j \times \omega$.) It is a local trivialization of $(\gamma \times \gamma) \big| _{{\C}P^{\infty} \times ({\C}P^{\infty} \setminus {\C}P^{\infty-1})}$, so it determines a local trivialization of $\xi \big| _{{\C}P^{\infty} \times ({\C}P^{\infty} \setminus {\C}P^{\infty-1})}$. This can be pulled back to a local trivialization $\sigma_j$ of $\T$ over $U'_j = f^{-1}(\hat{U}'_j \times ({\C}P^{\infty} \setminus {\C}P^{\infty-1})) \subseteq M \setminus N$.

Choose any $q = \varphi_k(p, t \theta) \in \pi^{-1}(U_k) \cap U'_j$. Then $f(q) = \tilde{\varphi}_k(f(p), t \theta) = ([a], \allowbreak [t \theta, b'])$, where $f(p) = ([a], [0, b']) \in \tilde{U}_k$. The fiber of $\gamma \times \gamma$ over this point can be identified with $\C^2$ by either $\tilde{\psi}'_k$ or $\tilde{\sigma}_j$. The identification by $\tilde{\psi}'_k$ sends $(x,y) \in \C^2$ to the point $(x a, (\frac{y}{\|(t \theta, b')\|} t \theta, \frac{y}{\|(t \theta, b')\|} b'))$. The other identification is given by the map $(x,y) \mapsto (x a, (\frac{y}{\| (1, c') \|}, \frac{y}{\| (1, c') \|} c'))$, where $[1, c'] = [t \theta, b']$. This implies that $b' =  t \theta c'$, and that $\|(t \theta, b')\| = t \| (1, c') \|$. Therefore the value of the transition map between $\tilde{\psi}'_k$ and $\tilde{\sigma}_j$ over this point (the map $\C^2 \rightarrow \C^2$ obtained by restricting $\tilde{\sigma}_j^{-1} \circ \tilde{\psi}'_k$ to this fiber, expressed as an element of $U(1) \times U(1)$) is $(1, \theta)$.

Similarly $F_q$ can be identified with $D^4$ by either $\psi'_k$ or $\sigma_j$. Let $(x, y)$ be coordinates on $F_q$ coming from the first identification, and $(x', y')$ be coordinates coming from the second. The transition map between $\psi'_k$ and $\sigma_j$ is determined by the transition map between $\tilde{\psi}'_k$ and $\tilde{\sigma}_j$, therefore $x' = x$ and $y' = \theta y$.

By the construction of $W'$, $W' \cap F_q = \left\{ (x',y') \mid 2x'y' = 1-|x'|^2 - |y'|^2 \right\}$. This can be expressed in the coordinates coming from $\psi'_k$ as $\left\{ (x,y) \mid 2x \theta y = 1-|x|^2 - |y|^2 \right\}$. This proves that the second condition is also satisfied.
\end{proof}

Now that we have a description of $W'$ around $N$, we can finish the construction of $W$ and prove that we get a manifold with boundary $B$.

For any fiber $F$ of $\T$ over a point of $N$, let 
\[
\begin{aligned}
\tilde{V} &= \left\{ (x,y) \in F \approx D^4 \mid 2|xy| \leq 1-|x|^2 - |y|^2 \right\} \\
&= \left\{ (x,y) \in F \approx D^4 \mid |x| + |y| \leq 1 \right\}.
\end{aligned}
\]
This is well-defined (invariant under the $(U(1) \times U(1))$-action). Let $W''$ be the union of the $\tilde{V}$s in all fibers over $N$, and let $W = W' \cup W''$.

$\tilde{V}$ is homeomorphic to $D^4$, the map $\tilde{v} : D^4 \rightarrow \tilde{V}$, $\tilde{v}(t(x,y)) = t \frac{1}{\sqrt{1+2|xy|}} (x,y)$ (where $t \in [0,1]$, and $(x,y) \in S^3$ \big[$(x,y) \in S^1$\big]) is a homeomorphism. 

\begin{prop}
$W$ is a (topological) submanifold of $T$.
\end{prop}

\begin{proof}
We need to find euclidean neigbourhoods of the points of $W''$ in $W$.  

Choose any $p \in N$. Let $G$ be the fiber of $\S$ over $p$, and fix an identification between $G$ and $D^2$, coming from some $\varphi_k$. Let $F_g$ be the fiber of $\T$ over $g \in G$, and $\tilde{V} = F_0 \cap W''$ be the fiber of $W''$ over $p$. $\psi'_k$ is used to identify $F_g$ with $D^4$.

A point of $\tilde{V}$ has a neighbourhood in $W$ that is the direct product of a neighbourhood of this point in $T \big| _G \cap W$, and a neighbourhood of $p$ in $N$. The latter is euclidean, so all we need to prove is that points of $\tilde{V}$ have euclidean neighbourhoods in $T \big| _G \cap W$.

\bek
By Lemma \ref{lem:loctriv} $F_0 \cap \overline{W'} = \left\{ (x,y) \in F_0 \approx D^4 \mid |xy| = 1-|x|^2 - |y|^2 \right\} = \partial \tilde{V}$ (here $\overline{W'}$ means the closure of $W'$). So $\interior \tilde{V}$ is open in $T \big| _G \cap W$, and it is also an open subset (hence a submanifold) in $F_0$, so its points do have euclidean neighbourhoods in $T \big| _G \cap W$.

\bek
Next consider a point $(x,y) \in \partial \tilde{V}$ with $x \neq 0$, $y \neq 0$. It has a euclidean half-space as a neighbourhood in $\tilde{V}$. In $T \big| _G \cap \overline{W'}$ it also has a half-space neighbourhood, which is the image of the following map. Its domain is the product of $[0,\varepsilon)$ (for a small $\varepsilon > 0$) and a small neighbourhood of $(x,y)$ in $\partial \tilde{V}$, and it sends $(t,(x',y'))$ (where $t \in [0,\varepsilon)$, $(x',y') \in \partial \tilde{V}$) to $(x',y')$ in the fiber $F_{t \frac{|x'y'|}{x'y'}}$. (Using Lemma \ref{lem:loctriv} we can check that this is a homeomorphism of a half-space onto a neighbourhood of $(x,y)$ in $T \big| _G \cap \overline{W'}$.) These two half-space neighbourhoods together form a euclidean neighbourhood of $(x,y)$ in $T \big| _G \cap W$.

\bek
Next take a point $(x,0) \in S^1 \times \{ 0 \} \subset \partial \tilde{V}$. This will be a boundary point of $W$, so we will construct a euclidean half-space neighbourhood of it in $T \big| _G \cap W$. 

Let $\R^i_+ = \left\{ (x_1,x_2,\dots,x_i) \in \R^i \mid x_1 \geq 0 \right\} $ be the $i$-dimensional half-space, and $C_i = \left\{ (x_1,x_2,\dots,x_i) \in \R^i \mid x_1 \geq \|(x_2,\dots,x_i)\| \right\} \subset \R^i_+$ be an $i$-dimensional cone.

Let $U$ be a neighbourhood of $x$ in $S^1$. \big[In Case 1 $U = \{ x \} \subset S^0$.\big] We will define a homeomorphism between a neighbourhood of $(x,0)$ in $T \big| _G \cap W$ and a neighbourhood of $(x,0,0,0)$ in $U \times \R^3_+$ \big[of $(x,0,0)$ in $U \times \R^2_+$\big].

First, for any $u \in U$, $u' \in S^1$, $s, t \in [0, \varepsilon)$, let $j'$ send the point $(u(1-s),u's)$ in the fiber $F_{t(uu')^{-1}}$ to $(u,s,u'(s+t)) \in U \times \R \times \R^2$ \big[$\in U \times \R \times \R$\big]. The domain of $j'$ is a neighbourhood of $(x,0)$ in $T \big| _G \cap \overline{W'}$ (the point $(u(1-s),u's) \in F_{t(uu')^{-1}}$ is in $W' \cap F_{t(uu')^{-1}}$ if $t > 0$, and it is in $\partial \tilde{V}$ if $t = 0$, and each point in a neighbourhood of $(x,0)$ is of this form). $j'$ is well-defined and injective because $(u_1(1-s_1),u'_1s_1) \in F_{t_1(u^{}_1u'_1)^{-1}}$ and $(u_2(1-s_2),u'_2s_2) \in F_{t_2(u^{}_2u'_2)^{-1}}$ are equal iff either $(u_1,u'_1,s_1,t_1) = (u_2,u'_2,s_2,t_2)$ or $u_1 = u_2$, $s_1 = s_2 = 0$, $u'_1 \neq u'_2$, and $t_1 = t_2 = 0$, and this is equivalent to $(u_1,s_1,u'_1(s_1+t_1)) = (u_2,s_2,u'_2(s_2+t_2))$. The image of $j'$ is in $U \times \overline{\R^3_+ \setminus C_3}$ \big[in $U \times \overline{\R^2_+ \setminus C_2}$\,\big], because $0 \leq s \leq |u'(s+t)|$, and it maps surjectively onto a neighbourhood of $(x,0,0,0)$. So $j'$ is a homeomorphism between a neighbourhood of $(x,0)$ in $T \big| _G \cap \overline{W'}$ and a neighbourhood of $(x,0,0,0)$ in $U \times \overline{\R^3_+ \setminus C_3}$.

Second, for any $u \in U$, $u' \in S^1$, $s, s' \in [0, \varepsilon)$, with $(1-s)+s' \leq 1$, let $j''$ send $(u(1-s),u's') \in \tilde{V}$ to $(u, s, u's') \in U \times \R \times \R^2$ \big[$\in U \times \R \times \R$\big]. This is well-defined and injective, and it maps into $U \times C_3$ \big[into $U \times C_2$\big], because $s \geq |u's'|$. So this is a homeomorphism between a neighbourhood of $(x,0)$ in $\tilde{V}$ and a neighbourhood of $(x,0,0,0)$ in $U \times C_3$.

The maps $j'$ and $j''$ are both defined on $\partial \tilde{V}$, and for any $u \in U$, $u' \in S^1$, $s \in [0, \varepsilon)$ they take the same value on $(u(1-s),u's) \in \tilde{V} \subset F_0$, namely $(u,s,u's)$. So together they form a homeomorphism $j$ between the neighbourhoods of $(x,0)$ in $T \big| _G \cap W$ and $U \times \R^3_+$.

\bek
Finally to prove that points in $\{ 0 \} \times S^1$ have euclidean neighbourhoods consider the map $(x,y) \mapsto (y,x)$ in all fibers of $T \big| _G$. It maps $T \big| _G \cap W$ onto itself and $S^1 \times \{ 0 \}$ onto $\{ 0 \} \times S^1$, so points of the latter also have euclidean neighbourhoods.
\end{proof}

We can also see from the above that $\partial W = B$. This finishes the proof of Theorem \ref{thm:top}.

\section{The smooth version} \label{sec:smooth}

\subsection{The construction of $W_2$}

Fix some $0 < \varepsilon_1 < \varepsilon_2 < 1$. Choose a smooth map $\ell : [0, 1] \rightarrow [0, 1]$ such that $\ell(t) = t$ if $t \in [0, \varepsilon_1]$, $\ell(t) = 1$ if $t \in [\varepsilon_2, 1]$, and $\ell$ is strictly increasing in $[\varepsilon_1, \varepsilon_2]$.

We define a subset $W''_2 \subset T \big| _S$ in the following way:

Let $q \in S$, $p = \pi(q) \in N$, and $F_q$ be the fiber of $\T$ over $q$. Choose a $k$ such that $p \in U_k$, and let $g = t \theta \in D^2$ (where $t \in [0, 1]$, $\theta \in S^1$) be the point that satisfies $q = \varphi_k(p, g)$. We use $\psi'_k$ to identify $F_q$ and $D^4$. Let
\[
\begin{aligned}
F_q \cap W''_2 &= \left\{ (x,y) \in F_q \approx D^4 \mid 2x\theta y = \ell(t)(1-|x|^2 - |y|^2) \right\} &\text{if $t>0$} \\
F_q \cap W''_2 &= \left\{ (x,y) \in F_q \approx D^4 \mid xy = 0 \right\} &\text{if $t=0$}
\end{aligned}
\]

\begin{prop} \label{prop:w''-well-def}
The set $W''_2$ is well-defined, ie.\ the subset $F_q \cap W''_2 \subset F_q$ defined above does not depend on the choice of $k$.
\end{prop}

\begin{proof} 
If $q \in N$, ie.\ $t=0$, then the set $\left\{ (x,y) \in F_q \approx D^4 \mid xy = 0 \right\}$ is independent of which $\psi'_k$ is used to identify $F_q$ and $D^4$, because it is $(U(1) \times U(1))$-invariant and $\T$ is a $(U(1) \times U(1))$-bundle.

If $q \not\in N$, ie.\ $t > 0$\,:

For some $L > 0$ consider the map $d_L : F_q \rightarrow F_q$,
\[
d_L(x,y) = \sqrt{\frac{L}{1-(1-L)(|x|^2+|y|^2)}}(x,y)\,.
\]
This makes sense (the denominator is always positive, since $1-L < 1$ and $0 \leq |x|^2+|y|^2 \leq 1$, and the image is in $F_q$ as we will see soon), and it is independent of which $\psi'_k$ is used to define coordinates in $F_q$. Indeed, it maps each point $(x,y)$ into itself multiplied by a scalar which depends only on $|x|^2+|y|^2$, and that is independent of $k$.

The map $d_{\frac{1}{L}} : F_q \rightarrow F_q$ is the inverse of $d_L$\,: 
\[
d_{\frac{1}{L}}(x,y) =  \sqrt{\frac{\frac{1}{L}}{1-(1-\frac{1}{L})(|x|^2+|y|^2)}}(x,y) = \sqrt{\frac{1}{L+(1-L)(|x|^2+|y|^2)}}(x,y)
\]
Indeed, both of $d_L$ and $d_{\frac{1}{L}}$ send a vector into a positive scalar multiple of itself, and the square of the norm of the vector is changed by the functions $h \mapsto \frac{Lh}{(L-1)h+1}$ and $h \mapsto \frac{h}{(1-L)h+L}$ respectively, and these linear fractional maps are inverses of each other. (They also map the interval $[0,1]$ into itself, so the image of $d_L$ is $F_q$\,.) Since both of $d_L$ and $d_{\frac{1}{L}}$ are smooth, they are diffeomorphisms.

\bek
We claim that $d_{\ell(t)}$ maps $F_q \cap W'$ to $F_q \cap W''_2$, independently of which $k$ is used to define the latter. This will prove that $F_q \cap W''_2$ is independent of $k$.

Choose any $k$ and use $\psi'_k$ to identify $F_q$ and $D^4$. By Lemma \ref{lem:loctriv}  
\[
F_q \cap W' = \left\{ (x,y) \in F_q \approx D^4 \mid 2x\theta y = 1-|x|^2 - |y|^2 \right\}.
\]

In order to prove our claim we first check that if $2x\theta y = r(1-|x|^2 - |y|^2)$ for some $\theta \in S^1$, $r > 0$, and $d_L(x,y) =(x',y')$, then $2x'\theta y' = Lr(1 - |x'|^2 - |y'|^2)$ (we will use the notation $c_L(x,y) = \sqrt{\frac{L}{1-(1-L)(|x|^2+|y|^2)}}$\,):
\begin{multline*}
2x'\theta y' = 2c_L(x,y)x\theta c_L(x,y)y = c_L(x,y)^2 2x\theta y = c_L(x,y)^2r(1 - |x|^2 - |y|^2) = \\
= \frac{Lr(1 - |x|^2 - |y|^2)}{1-(1-L)(|x|^2+|y|^2)} = Lr\left(\frac{1-(1-L)(|x|^2+|y|^2) - L(|x|^2+|y|^2)}{1-(1-L)(|x|^2+|y|^2)}\right) = \\
= Lr\left(1 - \frac{L(|x|^2+|y|^2)}{1-(1-L)(|x|^2+|y|^2)}\right) = Lr(1 - c_L(x,y)^2(|x|^2+|y|^2)) = \\
= Lr(1 - |x'|^2 - |y'|^2)
\end{multline*}

Hence if $(x,y) \in F_q \cap W'$, then $d_{\ell(t)}(x,y)\in F_q \cap W''_2$, so $d_{\ell(t)}(F_q \cap W') \subseteq F_q \cap W''_2$. Similarly $d_{\frac{1}{\ell(t)}}(F_q \cap W''_2) \subseteq F_q \cap W'$, so $F_q \cap W''_2 \subseteq d_{\ell(t)}(F_q \cap W')$, which implies that $d_{\ell(t)}(F_q \cap W') = F_q \cap W''_2$.

This works for each $k$, so our claim, and therefore the Proposition is proved.
\end{proof}

Let $W_2 = (W' \setminus T\big|_S) \cup W''_2$.

\begin{prop} \label{prop:W1smooth}
$W_2$ is a smooth submanifold of $T$.
\end{prop}

\begin{proof}
Let $S_{\varepsilon_2}$ denote (the total space of) the subbundle of $\S$ with fiber $D^2_{\varepsilon_2}$, the disk of radius $\varepsilon_2$. By definition of $W_2$ (using Lemma \ref{lem:loctriv}, and that $\ell(t) = 1$ if $t \in [\varepsilon_2, 1]$), $W_2 \cap T \big| _{M \setminus S_{\varepsilon_2}} = W' \cap T \big| _{M \setminus S_{\varepsilon_2}}$, therefore it is a smooth submanifold of $T \big| _{M \setminus S_{\varepsilon_2}}$.

Consider the map $d : T \big| _{S \setminus N} \rightarrow T \big| _{S \setminus N}$ which is defined by $d \big| _{F_q} = d_{\ell(t)}$, where $q = \varphi_k(p, g)$, $g = t \theta$. From the proof of Proposition \ref{prop:w''-well-def} we see that this is well-defined, and it is a diffeomorphism (recall also that $\ell$ is smooth). We also see that it maps $W' \cap T \big| _{S \setminus N}$ to $W_2 \cap T \big| _{S \setminus N}$, which proves that the latter is a smooth submanifold of $T \big| _{S \setminus N}$.

Now we only need to prove that $W_2$ is a smooth submanifold near $T \big|_N$. Let $p \in N$, and $G$ be the fiber of $S_{\varepsilon_1}$ over $p$. Near $F_p$ the pair $(T, W_2)$ is locally the product of a neighbourhood of $p$ in $N$ and the pair $(T \big| _G, W_2 \cap T \big| _G)$, so we need to prove that $W_2 \cap T \big| _G$ is a smooth submanifold in $T \big| _G$.

Let $k$ be such that $p \in U_k$, so $G = \varphi_k(\{ p \} \times D^2_{\varepsilon_1})$. For $g \in D^2_{\varepsilon_1}$ let $F_g = F_{\varphi_k(p, g)}$ be the appropriate fiber of $\T$. Let $\psi : D^2_{\varepsilon_1} \times D^4 \rightarrow T \big|_G$, $(g,x,y) \mapsto \psi'_k(\varphi_k(p,g),(x,y))$ be the trivialization of $T \big|_G$ coming from $\psi'_k$. If we use this trivialization, then 
\[
W_2 \cap T \big| _G = W''_2 \cap T \big| _G = \left\{ (g,x,y) \in T \big| _G \approx D^2_{\varepsilon_1} \times D^4 \mid 2xy = \bar{g}(1-|x|^2 - |y|^2) \right\} \text{\,.}
\]
(By the definition of $W''_2$, $W''_2 \cap F_0 = \left\{ (x,y) \in F_0 \mid xy = 0 \right\}$, and if $g = t \theta$, $0 < t < \varepsilon_1$ then $W''_2 \cap F_g = \left\{ (x,y) \in F_g \mid 2xy = t \theta^{-1} (1-|x|^2 - |y|^2) \right\}$, and in this case $t \theta^{-1} = \bar{g}$.)

Consider the map $w : \R^6 \rightarrow \R^2$, 
\begin{multline*}
w(g_1,g_2,x_1,x_2,y_1,y_2) = (2x_1y_1-2x_2y_2-g_1(1-x_1^2-x_2^2-y_1^2-y_2^2) , \\
2x_1y_2+2x_2y_1+g_2(1-x_1^2-x_2^2-y_1^2-y_2^2))
\end{multline*}

$W''_2 \cap T \big| _G = D^2_{\varepsilon_1} \times D^4 \cap w^{-1}(0,0)$, so it is enough to prove that each point in $W''_2 \cap T \big| _G$ is a regular point of $w$.
\begin{multline*}
\mathrm{d}w = 
\bigg[
\begin{matrix}
-(1-x_1^2-x_2^2-y_1^2-y_2^2) & 0 & 2y_1+2g_1x_1 \\
0 & (1-x_1^2-x_2^2-y_1^2-y_2^2) & 2y_2-2g_2x_1
\end{matrix} \\
\begin{matrix}
-2y_2+2g_1x_2 & 2x_1+2g_1y_1 & -2x_2+2g_1y_2 \\
2y_1-2g_2x_2 & 2x_2-2g_2y_1 & 2x_1-2g_2y_2
\end{matrix}
\bigg]
\end{multline*}

If $|x|^2 + |y|^2 < 1$, then the first two columns of the matrix are linearly independent, therefore $(g,x,y)$ is a regular point. If $|x|^2 + |y|^2 = 1$, and $(g,x,y) \in W''_2 \cap T \big| _G$, then $xy = 0$, therefore $x = 0$, $|y| = 1$, or $y = 0$, $|x| = 1$. In the first case the third and fourth column of the matrix are two orthogonal vectors of length 2, therefore they are linearly independent. In the second case the same holds for the last two columns. Therefore the rank of $\mathrm{d}w$ is 2 at each point of $W''_2 \cap T \big| _G$, so it is a smooth submanifold.

\quad\!\!\!\!\!\!
\big[In Case 1 we use the map $w : \R^3 \rightarrow \R$, $w(g,x,y) = 2xy-g(1-x^2-y^2)$. Then $\mathrm{d}w = [1-x^2-y^2, 2y+2gx, 2x+2gy]$, and this is never 0. If $x^2+y^2<1$ then the first, if $x=0$, $|y|=1$, then the second, and if $y=0$, $|x|=1$ then the third entry is non-zero.\big]
\end{proof}

\begin{prop}
The boundary of $W_2$ is $\partial W_2 = B$.
\end{prop}

\begin{proof}
For any fiber $F$ of $\T$
\[
\begin{aligned}
\partial W_2 \cap F = W_2 \cap \partial F &= \left\{ (x,y) \in D^4 \mid xy = 0, |x|^2 + |y|^2 = 1 \right\} = \\
 &= \left\{ (x,y) \in D^4 \mid (y = 0 \text{ and } |x| = 1) \text{ or } (x = 0 \text{ and } |y| = 1) \right\} = \\
 &= S^1 \times \{ 0 \} \bigsqcup \{ 0 \} \times S^1 = B \cap F \text{\,.}
\end{aligned}
\]
\end{proof}

\begin{prop}
$W_2$ is co-oriented (in Case 2).
\end{prop}

\begin{proof}
First, $W_2 \cap T \big|_{M \setminus N}$ is co-oriented (it is the image of $W'$ under a diffeomorphism of $T \big|_{M \setminus N}$, and $W'$ is co-oriented, as we saw in Section \ref{ssec:sp}). This is enough, because $W_2 \cap T \big|_N$ is the union of two codimension-2 submanifolds in $W_2$ (they are the subbundles of $\T$ over $N$ with fibers $D^2 \times \{ 0 \}$ and $\{ 0 \} \times D^2$), so the co-orientation of $W_2 \cap T \big|_{M \setminus N}$ extends to $W_2$.
\end{proof}

\subsection{A homeomorphism between $W$ and $W_2$} \label{ssec:homeo}

We will show a homeomorphism between $(T \big| _S \cap W)$ and $W''_2$ that is identical on $\partial (T \big| _S \cap W)$. It immediately extends to a homeomorphism $W \rightarrow W_2$ that is identical on $\partial W$.

First we define a subset $W''_0 \subset T \big| _S$. 

For any $p \in N$ choose a $k$ with $p \in U_k$, use $\varphi_k$ to identify $D^2$ with the fiber $G$ of $\S$ over $p$, and use $\psi'_k$ to identify $D^2 \times D^4$ with $T \big| _{G}$. Let $F_g$ be the fiber of $\T$ over $g \in G$. Let
\[
T \big| _{G} \cap W''_0 = \left\{ (z,x,y) \in T \big| _{G} \approx D^2 \times D^4 \mid 2xy = \bar{z}(1-|x|^2 - |y|^2) \right\} 
\]

\begin{prop}
The set $W''_0$ is well-defined, and it is homeomorphic to $W''_2$.
\end{prop}

\begin{proof}
From the definition we see that $F_0 \cap W''_0 = \left\{ (x,y) \in F_0 \approx D^4 \mid xy = 0 \right\}$, and $F_{t\theta} \cap W''_0 = \left\{ (x,y) \in F_{t\theta} \approx D^4 \mid 2x\theta y = t(1-|x|^2 - |y|^2) \right\}$ (if we use $\varphi_k$ and $\psi'_k$). Let $d : T \big| _{G} \rightarrow T \big| _{G}$, such that $d \big|_{F_0} = \id$ and $d \big|_{F_{t\theta}} = d_{\frac{t}{\ell(t)}}$ (see the proof of Proposition \ref{prop:w''-well-def} for the definition of $d_L$). 

We saw that $d$ is well-defined (independent of $k$). It is continuous, because $\frac{t}{\ell(t)} = 1$ if $t \leq \varepsilon_1$ and $d_1 = \id$. The calculation in the proof of Proposition \ref{prop:w''-well-def} shows that $d \bigl( T \big| _{G} \cap W''_2 \bigr) = T \big| _{G} \cap W''_0$\,. Since $W''_2$ is well-defined, this is also true for $W''_0$. And if we consider all these maps $d$ for each $p \in N$, together they form a homeomorphism (in fact a diffeomorphism) between $W''_2$ and $W''_0$.
\end{proof}

We will construct a homeomorphism $I : T \big| _S \cap W \rightarrow W''_0$ which is identical on the boundary. This, composed with the homeomorphism we just defined will be a suitable homeomorphism $T \big| _S \cap W \rightarrow W''_2$.

\bek
Let $p \in N$, and let $G$ be the fiber of $\S$ over $p$. We will construct a homeomorphism $i : T \big| _G \cap W \rightarrow T \big| _G \cap W''_0$, this will be made up of two pieces, $i'$ and $i''$.

Choose a $k$ with $p \in U_k$, and use $\varphi_k$ and $\psi'_k$ for the identifications $G \approx D^2$ and $T \big| _{G} \approx D^2 \times D^4$ respectively. Let $F_g$ be the fiber of $\T$ over $g \in G$, and $\tilde{V} = F_0 \cap W$.

Let the maps $a, b : T \big| _G \cap \overline{W'} \rightarrow D^2$ and $i' : T \big| _G \cap \overline{W'} \rightarrow T \big| _{G} \approx D^2 \times D^4$ be defined by the following formulas. (Here $t \in [0,1]$, $\theta \in S^1$, $x,y \in D^2$, and it follows from Lemma \ref{lem:loctriv} that $(x,y)\in F_{t\theta} \cap (T \big| _G \cap \overline{W'})$ iff either $t=0$, $|x|+|y|=1$ or $t>0$, $2x\theta y = 1-|x|^2-|y|^2$.)
\[
\begin{aligned}
a(t\theta,x,y) &= \left( (1-t)(1-|x|)(|x|^2-1)+1 \right) x \\
b(t\theta,x,y) &= \left( (1-t)(1-|y|)(|y|^2-1)+1 \right) y \\
i'(t\theta,x,y) &= 
\begin{cases}
\left( \frac{2\overline{a(t\theta,x,y)b(t\theta,x,y)}}{1-|a(t\theta,x,y)|^2-|b(t\theta,x,y)|^2}, a(t\theta,x,y), b(t\theta,x,y) \right) & \text{if $|x|^2+|y|^2<1$} \\
(t\theta,x,y) & \text{if $|x|^2+|y|^2=1$}
\end{cases}
\end{aligned}
\]
Note that (since $|x|+|y|=1$) $|x|^2+|y|^2=1$ holds iff $|x|=1$, $y=0$ or $|y|=1$, $x=0$. It is also equivalent to $1-|a(t\theta,x,y)|^2-|b(t\theta,x,y)|^2 = 0$ (because $|a(t\theta,x,y)| \leq |x|$ and $|b(t\theta,x,y)| \leq |y|$).

\begin{prop} \label{prop:i'-well-def} 
The map $i'$ is well-defined (ie.\ independent of $k$).
\end{prop}

\begin{proof}
Assume that $\psi'_{k_1}(\varphi_{k_1}(p,t_1\theta_1),(x_1,y_1)) = \psi'_{k_2}(\varphi_{k_2}(p,t_2\theta_2),(x_2,y_2))$, for some $k_1$, $k_2$, then we want to prove that $\psi'_{k_1}(\varphi_{k_1}(p,z_1),(a_1,b_1)) = \psi'_{k_2}(\varphi_{k_2}(p,z_2),(a_2,b_2))$, where $a_i = a(t_i\theta_i,x_i,y_i)$, $b_i = b(t_i\theta_i,x_i,y_i)$, and $z_i = \frac{2\overline{a_ib_i}}{1-|a_i|^2-|b_i|^2}$ if $|x_i|^2+|y_i|^2<1$ and $z_i = t_i\theta_i$ if $|x_i|^2+|y_i|^2=1$.

Our assumption means that $\varphi_{k_1}(p,t_1\theta_1) = \varphi_{k_2}(p,t_2\theta_2)$ (therefore $t_1=t_2$ and $\theta_2 = \varphi_{k_1,k_2}(p)(\theta_1)$), and that $\psi_{k_1}(p,(x_1,y_1)) = \psi_{k_2}(p,(x_2,y_2))$. Let $(\alpha, \beta) = \psi_{k_1,k_2}(p) \in U(1) \times U(1)$, then $x_2 = \alpha x_1$, and $y_2 = \alpha^{-1} \beta y_1$. This implies that $a_2 = \alpha a_1$ and $b_2 = \alpha^{-1} \beta b_1$, therefore $\psi_{k_1}(p, \allowbreak (a_1,b_1)) = \psi_{k_2}(p,(a_2,b_2))$.

We can also see that if $|x_i|^2+|y_i|^2<1$, then $z_2 = \overline{\beta} z_1$ (note that $|x_1|^2+|y_1|^2 = |x_2|^2+|y_2|^2$). By the first part of Lemma \ref{lem:loctriv} $\varphi_{k_1,k_2}(p) = \overline{\pr_2(\psi_{k_1,k_2}(p))} = \overline{\beta}$, therefore $\theta_2 = \overline{\beta} \theta_1$, so $z_2 = \overline{\beta} z_1$ holds even if $|x_i|^2+|y_i|^2=1$. Therefore $\varphi_{k_1}(p,z_1) = \varphi_{k_2}(p,z_2)$.

From the definition of $\psi'_k$ we see that the two equalities we have proved imply that $\psi'_{k_1}(\varphi_{k_1}(p,z_1),(a_1,b_1)) = \psi'_{k_2}(\varphi_{k_2}(p,z_2),(a_2,b_2))$.
\end{proof}

\begin{prop}
The map $i'$ is continuous.
\end{prop}

\begin{proof}
We need to prove this at points $(x,y)$, where $|x|^2+|y|^2=1$, ie.\ $|x|=1$, $y=0$ or $|y|=1$, $x=0$. The maps $a$ and $b$ are everywhere continuous (and smooth), and at such points $a(t\theta,x,y)=x$ and $b(t\theta,x,y)=y$.

Now fix some $x_0$ of absolute value 1, and $t_0\theta_0 \in D^2$, we will consider the point $(t_0\theta_0,x_0,0)$ (points of the other type can be handled similarly). We have to prove that the following expression tends to $t_0\theta_0$ as $(t\theta,x,y) \rightarrow (t_0\theta_0,x_0,0)$:
\[
\!\!\!\!\!\!\!\!\!\!\!\!\!\!\!\!\!\!\!\!\!\!\!\!\!\!\!\!\!\!\!\!\!\!\!\!\!\!\!\!\!\!\!\!\!\!\!\!\!\!\!\!\!\!\!\!\!\!\!\!\!\!\!\!\!\!\!\!\!\!\!\!\!\!\!\!\!\!\!\!\!\!\!\!\!\!\!\!\!\!\!\!\!\!\!\!\!\!\!\!\!\!\!\!\!\!\!\!\!\!\!\!\!\!\!\!\!\!\!\frac{2\overline{a(t\theta,x,y)b(t\theta,x,y)}}{1-|a(t\theta,x,y)|^2-|b(t\theta,x,y)|^2} = 
\]
\[
= \frac{2\overline{xy} \bigl((1-t)(1-|x|) (|x|^2-1)+1 \bigr) \bigl((1-t)(1-|y|) (|y|^2-1)+1 \bigr)} {1-|x|^2 \bigl((1-t)(1-|x|) (|x|^2-1)+1 \bigr)^2-y\bar{y} \bigl( (1-t)(1-|y|) (|y|^2-1)+1 \bigr)^2} = 
\]
\[
\!\!\!\!\!\!\!\!\!\!\!\!\!\!\!\!\!\!\!\!\!\!= \frac{2 \bigl( (1-t)(1-|x|) (|x|^2-1)+1 \bigr) \bigl( (1-t)(1-|y|) (|y|^2-1)+1 \bigr) }{\frac{1- |x|((1-t)(1-|x|) (|x|^2-1)+1)}{\overline{xy}}\bigl(1+|x|\bigl( (1-t)(1-|x|) (|x|^2-1)+1 \bigr)\bigr)-\frac{y}{\overline{x}}\bigl( (1-t)(1-|y|) (|y|^2-1)+1 \bigr)^2}
\]

Here $((1-t)(1-|x|) (|x|^2-1)+1) \rightarrow 1$, and $((1-t)(1-|y|) (|y|^2-1)+1) \rightarrow t_0$, $(1+|x|( (1-t)(1-|x|) (|x|^2-1)+1 )) \rightarrow 2$, and $\frac{y}{\overline{x}}( (1-t)(1-|y|) (|y|^2-1)+1)^2 \rightarrow 0$. Since $|y|=1-|x|$ for all points $(t\theta,x,y) \in T \big| _G \cap \overline{W'}$, 
\begin{multline*}
\left| \frac{1- |x|((1-t)(1-|x|) (|x|^2-1)+1)}{\overline{xy}} \right| = \frac{|1 - |x| - |x|(1-t)(1-|x|)(|x|^2-1)|}{\left|x\right| (1 - \left| x \right|)} = \\
= \left| \frac{1}{|x|} - (1-t)(|x|^2-1) \right| \rightarrow 1
\end{multline*}
Therefore if $t_0=0$, then the whole expression tends to 0. If $t_0 \neq 0$, then $\theta$ is well-defined for points $(t\theta,x,y)$ in a neighbourhood of $(t_0\theta_0,x_0,0)$, and it is the argument of $\overline{xy}$, therefore $\frac{1- |x|((1-t)(1-|x|) (|x|^2-1)+1)}{\overline{xy}} \rightarrow \overline{\theta_0}$.
\end{proof}

\begin{prop}
The map $i'$ is injective.
\end{prop}

\begin{proof}
$i'$ restricted to points with $|x|^2+|y|^2=1$ is the identity map, so it is injective here. If $|x|^2+|y|^2<1$ then $|a(t\theta,x,y)|^2+|b(t\theta,x,y)|^2 < 1$, so $i'$ can't take the same value on such a point as on a point of the first type. So it is enough to consider points of the second type. We will prove that $(a,b) : T \big| _G \cap \overline{W'} \rightarrow D^4$ is injective on such points, this will imply our statement. 

For a point $(t\theta,x,y) \in T \big| _G \cap \overline{W'}$ with $|x|^2+|y|^2<1$ (ie.\ $xy \neq 0$), $\theta$ is the argument of $\overline{xy}$, so it is uniquely determined by $x$ and $y$. So what we want to prove is that if $(a,b)(t_1\theta_1,x_1,y_1)=(a,b)(t_2\theta_2,x_2,y_2)$, then $t_1=t_2$, $x_1=x_2$, and $y_1=y_2$. Since $a(t\theta,x,y)$ is a positive scalar multiple of $x$, and $b(t\theta,x,y)$ is a positive scalar multiple of $y$, that equality holds iff $x_2$ and $y_2$ are positive scalar multiples of $x_1$ and $y_1$ respectively and $|a(t_1\theta_1,x_1,y_1)|=|a(t_2\theta_2,x_2,y_2)|$ and $|b(t_1\theta_1,x_1,y_1)|=|b(t_2\theta_2,x_2,y_2)|$. 

Therefore it is enough to prove that the last two equalities imply that $t_1 = t_2$ and $|x_1|=|x_2|$ (and $|y_1|=|y_2|$, but this follows from $|x_1|=|x_2|$, because $|y| = 1 - |x|$ for points of  $T \big| _G \cap \overline{W'}$).

Let $(a^*,b^*) : [0,1] \times (0,1) \rightarrow (0,1) \times (0,1)$, 
\[
(a^*,b^*)(t,x) = \left( ((1-t)(1-x)(x^2-1)+1)x , ((1-t)x((1-x)^2-1)+1)(1-x) \right).
\]
Then $|a(t\theta,x,y)| = a^*(t,|x|)$, and $|b(t\theta,x,y)| = b^*(t,|x|)$. So we need to prove that $(a^*,b^*)$ is injective. (Recall that we are considering points with $|x|^2+|y|^2<1$, so $|x| \in (0,1)$ instead of $[0,1]$.)

Suppose $(a^*,b^*)(t,x) =(A,B)$ for some fixed $(A,B) \in (0,1) \times (0,1)$, then 
\[
(1-t)=\frac{A-x}{(1-x)(x^2-1)x}=\frac{B+x-1}{x(x^2-2x)(1-x)} \tag{$**$}
\]
This implies that $(A-x)(x^2-2x)=(B+x-1)(x^2-1)$, therefore $x$ satisfies the equation $0 = 2x^3+(B-A-3)x^2+(2A-1)x+(1-B)$. The value of this degree-3 polinomial is $(1-B)>0$ in 0 and $(A-1)<0$ in 1, so it has one root in each of the intervals $(-\infty,0)$, $(0,1)$, $(1, \infty)$. Since $x$ is in $(0,1)$, the pair $(A,B)$ uniquely determines $x$, and also $t$ by ($**$). Therefore $(a^*,b^*)$ is injective.
\end{proof}

Let $c, d : \tilde{V} \rightarrow D^2$ and $i'' : \tilde{V} \rightarrow T \big| _{G} \approx D^2 \times D^4$ be defined by the following formulas. (Here $x,y \in D^2$ such that $|x|+|y| \leq 1$.)
\[
\begin{aligned}
c(x,y) &= \left( \frac{1+|y|^2-|x|^2}{2} (|x|^2-1)+1 \right) x \\
d(x,y) &= \left( \frac{1+|x|^2-|y|^2}{2} (|y|^2-1)+1 \right) y \\
i''(x,y) &= 
\begin{cases}
\left( \frac{2\overline{c(x,y)d(x,y)}}{1-|c(x,y)|^2-|d(x,y)|^2}, c(x,y), d(x,y) \right) & \text{if $|x|^2+|y|^2<1$} \\
(0,x,y) & \text{if $|x|^2+|y|^2=1$}
\end{cases}
\end{aligned}
\]

\begin{prop}
The map $i''$ is well-defined (ie.\ independent of $k$).
\end{prop}

\begin{proof}
The proof is analogous to that of Proposition \ref{prop:i'-well-def}. Suppose $\psi_{k_1}(p,(x_1,y_1)) = \psi_{k_2}(p,(x_2,y_2))$ for some $k_1$, $k_2$, then we want to prove that $\psi'_{k_1}(\varphi_{k_1}(p,z_1),(c_1,d_1)) = \psi'_{k_2}(\varphi_{k_2}(p,z_2),(c_2,d_2))$, where $c_i = c(x_i,y_i)$, $d_i = d(x_i,y_i)$, and $z_i = \frac{2\overline{c_id_i}}{1-|c_i|^2-|d_i|^2}$ if $|x_i|^2+|y_i|^2<1$ and $z_i = 0$ if $|x_i|^2+|y_i|^2=1$.

Let $(\alpha, \beta) = \psi_{k_1,k_2}(p) \in U(1) \times U(1)$, this means that $x_2 = \alpha x_1$, and $y_2 = \alpha^{-1} \beta y_1$. This implies that $c_2 = \alpha c_1$, and $d_2 = \alpha^{-1} \beta d_1$, so $\psi_{k_1}(p,(c_1,d_1)) = \psi_{k_2}(p,(c_2,d_2))$.

Another consequence is that $z_2 = \overline{\beta} z_1$. By the first part of Lemma \ref{lem:loctriv}, $\varphi_{k_1,k_2}(p) = \overline{\pr_2(\psi_{k_1,k_2}(p))} = \overline{\beta}$, therefore $\varphi_{k_1}(p,z_1) = \varphi_{k_2}(p,z_2)$.
\end{proof}

\begin{prop}
The map $i''$ is continuous.
\end{prop}

\begin{proof}
We need to prove this at points $(x,y)$, where $|x|^2+|y|^2=1$. The maps $c$ and $d$ are everywhere continuous, and at such points $c(x,y)=x$ and $d(x,y)=y$.

Now fix some $x_0$ of absolute value 1, we will consider the point $(x_0,0)$. We have to prove that the following expression tends to 0 as $(x,y) \rightarrow (x_0,0)$:
\[
\!\!\!\!\!\!\!\!\!\!\!\!\!\!\!\!\!\frac{2\overline{c(x,y)d(x,y)}}{1-|c(x,y)|^2-|d(x,y)|^2} = \frac{2\overline{xy}\left( \frac{1+|y|^2-|x|^2}{2} (|x|^2-1)+1 \right) \left( \frac{1+|x|^2-|y|^2}{2} (|y|^2-1)+1 \right) }{1-|x|^2\left( \frac{1+|y|^2-|x|^2}{2} (|x|^2-1)+1 \right)^2-y\bar{y}\left( \frac{1+|x|^2-|y|^2}{2} (|y|^2-1)+1 \right)^2} = 
\]
\[
\!\!\!\!\!\!\!\!\!\!\!\!\!\!\!\!\!= \frac{2\bar{x}\left( \frac{1+|y|^2-|x|^2}{2} (|x|^2-1)+1 \right)} {\frac{1-|x|\left( \frac{1+|y|^2-|x|^2}{2} (|x|^2-1)+1 \right)}{\bar{y}\left( \frac{1+|x|^2-|y|^2}{2} (|y|^2-1)+1 \right)} \left(1+|x|\left( \frac{1+|y|^2+|x|^2}{2} (|x|^2-1)+1 \right)\right) - y\left( \frac{1+|x|^2-|y|^2}{2} (|y|^2-1)+1 \right)}
\]
\quad

$2\bar{x}\left( \frac{1+|y|^2-|x|^2}{2} (|x|^2-1)+1 \right) \rightarrow 2\overline{x_0}$ and $\left(1+|x|\left( \frac{1+|y|^2+|x|^2}{2} (|x|^2-1)+1 \right)\right) \rightarrow \nolinebreak 2$ \break and $y\left( \frac{1+|x|^2-|y|^2}{2} (|y|^2-1)+1 \right) \rightarrow 0$ hold trivially. 

$\frac{1-|x|\left( \frac{1+|y|^2-|x|^2}{2} (|x|^2-1)+1 \right)}{\left( \frac{1+|x|^2-|y|^2}{2} (|y|^2-1)+1 \right)}$ is at least 1, because $1-|x|\left( \frac{1+|y|^2-|x|^2}{2} (|x|^2-1)+1 \right) \geq \left( \frac{1+|x|^2-|y|^2}{2} (|y|^2-1)+1 \right)$, because $2-|x|\bigl( (1+|y|^2-|x|^2) (|x|^2-1)+2 \bigr) - \bigl( (1+ \break |x|^2 -|y|^2) (|y|^2-1)+2 \bigr)$ $=$ $1-|x|+|x|^2-2|y|^2-2|x|^3+|x||y|^2-|x|^2|y|^2+|y|^4 + \break |x|^5- |x|^3|y|^2$ $=$ $(1-|x|-|y|^2)(1+|x|^2+|x|^3-|y|^2)-2|x|^3+|x|^4+|x|^5$ $\geq$ $(1- \break |x|-(1-|x|)^2)(1+|x|^2+|x|^3-(1-|x|)^2)-2|x|^3+|x|^4+|x|^5 = (|x|-|x|^2)(2|x|+|x|^3)-2|x|^3+|x|^4+|x|^5=2|x|^2-4|x|^3+2|x|^4=2(|x|-|x|^2)^2 \geq 0$. (For the first in\-equality we used that $|x| + |y| \leq 1$ for all $(x,y) \in \tilde{V}$, therefore the left-hand side is a decreasing function of $|y|$, and $|y| \leq 1 - |x|$.)

This implies that the absolute value of the factor $\frac{1-|x|\left( \frac{1+|y|^2-|x|^2}{2} (|x|^2-1)+1 \right)}{\bar{y}\left( \frac{1+|x|^2-|y|^2}{2} (|y|^2-1)+1 \right)}$ tends to infinity (because this is true for $\frac{1}{\overline{y}}$). This proves our statement.
\end{proof}

\begin{prop}
The map $i''$ is injective.
\end{prop}

\begin{proof}
We will prove that $(c,d) : \tilde{V} \rightarrow D^4$ is injective, this implies our statement. Since $c(x,y)$ is a positive scalar multiple of $x$ (actually the coefficient of $x$ is non-negative but it can be 0 only if $x=0$), and $d(x,y)$ is a positive scalar multiple of $y$, $(c,d)(x_1,y_1)=(c,d)(x_2,y_2)$ holds iff $x_2$ and $y_2$ are positive scalar multiples of $x_1$ and $y_1$ respectively, and $|c(x_1,y_1)|=|c(x_2,y_2)|$ and $|d(x_1,y_1)|=|d(x_2,y_2)|$. Therefore it is enough to prove that if the last two equations hold, then $|x_1|=|x_2|$ and $|y_1|=|y_2|$.

Let $(c^*,d^*) : \{ (x,y) \in [0,1] \times [0,1] \mid x+y \leq 1 \} \rightarrow \R^2$, 
\[
(c^*,d^*)(x,y) = \left( \left( \frac{1+y^2-x^2}{2} (x^2-1)+1 \right) x , \left( \frac{1+x^2-y^2}{2} (y^2-1)+1 \right) y \right).
\]
Then $|c(x,y)|= c^*(|x|,|y|)$, and $|d(x,y)|= d^*(|x|,|y|)$, so what we need to prove is that $(c^*,d^*)$ is injective. Let
\[
D = \mathrm{d}(c^*,d^*) = 
\begin{bmatrix} 
\frac{1+6x^2-5x^4+(3x^2-1)y^2}{2} & x(x^2-1)y \\
y(y^2-1)x & \frac{1+6y^2-5y^4+(3y^2-1)x^2}{2}
\end{bmatrix}
\]
Let $D^S = \frac{D+D^T}{2}$ be the symmetric part of $D$. Its $(1,1)$-entry is $\frac{6x^2-5x^4}{2}+\frac{1+(3x^2-1)y^2}{2} \geq 0$, and it is $0$ iff $x=0$, $y=1$. Its determinant is $\frac{1}{4} \bigl( 1+6x^2-5x^4+(3x^2-1)y^2 \bigr) \bigl( 1+6y^2-5y^4+(3y^2-1)x^2 \bigr) - x^2y^2 \big( \frac{x^2+y^2}{2}-1 \big) ^2 = \frac{1}{4} \bigl( 1 + 5 x^2 + 5 y^2 - 11 x^4 + 39 x^2 y^2 - \break  11 y^4 + 5 x^6 - 11 x^4 y^2 - 11 x^2 y^4 + 5 y^6 - 16 x^6 y^2 + 32 x^4 y^4 - 16 x^2 y^6 \bigr) = \frac{1}{4} \bigl( (1 - x^4 - y^4) \break + 5(x - x^3)^2 + 5(y - y^3)^2 + 11x^2y^2(1-x^2-y^2) + 16x^2y^2(1-x^4-y^4) + 12x^2y^2 + \break 32x^4y^4 \bigr) \geq 0$, and it is $0$ if $x=1$, $y=0$ or  $x=0$, $y=1$. So $D^S$, and therefore $D$ too, are positive definite in each point except these two.

Suppose $(x_1,y_1) \neq (x_2,y_2)$ are two points in the domain of $(c^*,d^*)$, let $\Delta = (x_2,y_2) - (x_1,y_1)$ be their difference, and $\gamma : [0,1] \rightarrow \R^2$, $\gamma(t) = (x_1,y_1) + t\Delta$ be the line segment between them (it is in the domain of $(c^*,d^*)$, because this domain is convex). $(c^*,d^*) \circ \gamma$ is a path between the images of the points, so $\int_0^1 ((c^*,d^*)(\gamma(t)))'\mathrm{d}t = (c^*,d^*)(x_2,y_2) - (c^*,d^*)(x_1,y_1)$. The integrand is $D(\gamma(t))\Delta$, so $\left< \Delta, (c^*,d^*)(x_2,y_2) - \right. \allowbreak \left. (c^*,d^*)(x_1,y_1) \right> = \left< \Delta, \int_0^1 D(\gamma(t))\Delta \mathrm{d}t \right> =  \int_0^1 \left< \Delta, D(\gamma(t))\Delta \right> \mathrm{d}t$, which is positive, be\-cause the matrix $D(\gamma(t))$ is positive definite for all but at most two values of $t$. Therefore $(c^*,d^*)(x_2,y_2) - (c^*,d^*)(x_1,y_1) \neq 0$, so $(c^*,d^*)$ is injective.
\end{proof}

Both of $i'$ and $i''$ are defined on $\left( T \big| _G \cap \overline{W'} \right) \cap \tilde{V} = \partial \tilde{V}$, and they coincide, because if $(x,y) \in \partial \tilde{V}$ (ie.\ $|x|+|y| = 1$), then $a(0,x,y) = \left( (1-|x|)(|x|^2-1)+1 \right) x = \break \left( \frac{1+|y|^2-|x|^2}{2} (|x|^2-1)+1 \right) x = c(x,y)$, similarly $b(0,x,y) = d(x,y)$, hence $i'(0,x,y)=i''(x,y)$. Therefore $i'$ and $i''$ together form a continuous map $i : T \big| _G \cap W \rightarrow T \big| _{G}$\,. This $i$ is well-defined (does not depend on the choice of a local trivialization), because this is true for $i'$ and $i''$.

\begin{prop}
The map $i$ is identical on the boundary, and the image of $i$ is $T \big| _G \cap W''_0$.
\end{prop}

\begin{proof}
First we prove that $i$ is identical on the boundary. There are no boundary points in $\interior \tilde{V}$. The point $(t\theta,x,y) \in T \big| _G \cap \overline{W'}$ is a boundary point iff $|x|^2+|y|^2=1$ or $t=1$. In the first case $i(t\theta,x,y) = i'(t\theta,x,y)=(t\theta,x,y)$ by definition. In the second case $a(\theta,x,y) = x$, $b(\theta,x,y) = y$, and $i(\theta,x,y) = i'(\theta,x,y) = \bigl( \frac{2\overline{xy}}{1-|x|^2-|y|^2},x,y \bigr) =(\theta,x,y)$, because for these points $2x\theta y = 1-|x|^2-|y|^2$, hence $\frac{2xy}{1-|x|^2-|y|^2} = \bar{\theta}$. This proves the first statement.

Next we prove that $i$ maps into $W''_0$. Since the boundary of $T \big| _G \cap W''_0$ coincides with that of $T \big| _G \cap W$, we only need to check this for interior points (so $|x|^2+|y|^2<1$ here). By definition $W''_0$ contains the points $(z,x,y)$ with  $2xy = \bar{z}(1-|x|^2 - |y|^2)$, which is equivalent to $z = \frac{2\overline{xy}}{1-|x|^2-|y|^2}$. Both of $i'(t\theta,x,y)$ and $i''(x,y)$ satisfy this condition.

Finally we prove that $i$ is surjective onto $T \big| _G \cap W''_0$. Let $T \big| _G \cap W = P_1$ and $T \big| _G \cap W''_0 = P_2$, these are 4-manifolds with boundary. \big[In Case 1 $\dim P_1 = \dim P_2 = 2$, so all dimensions are shifted by $(-2)$.\big] Assume indirectly that $q \in \interior P_2$ is not in the image of $i$. From the exact sequence of the triple $(P_2 , P_2 \setminus \{ q \} , \partial P_2)$ (using that $P_2$ is connected) we see that $H_4(P_2 \setminus \{ q \} , \partial P_2) = 0$. In the exact sequence of the pair $(P_2 \setminus \{ q \} , \partial P_2)$ this implies that the homomorphism $H_3(\partial P_2) \rightarrow H_3(P_2 \setminus \{ q \})$ is injective, so the image of the generator of $H_3(\partial P_2) \cong \Z$ is non-zero.

On the other hand the homomorphism $H_3(\partial P_1) \rightarrow H_3(P_1)$ induced by the inclusion maps the generator to 0. Since $i$ is identical on the boundary, $i_* : H_*(\partial P_1) \rightarrow H_*(\partial P_2)$ is an isomorphism. So the commutativity of the following diagram implies that the homomorphism $H_3(\partial P_2) \rightarrow H_3(i(P_1))$ (which is induced by the inclusion) maps the generator to 0.
\[
\begin{CD}
H_3(\partial P_1) @>>> H_3(P_1) \\
@VVi_*V @VVi_*V \\
H_3(\partial P_2) @>>> H_3(i(P_1))
\end{CD}
\]
Since $q \not\in i(P_1)$, the homomorphism $H_3(\partial P_2) \rightarrow H_3(P_2 \setminus \{ q \})$ is in fact a composition $H_3(\partial P_2) \rightarrow H_3(i(P_1)) \rightarrow H_3(P_2 \setminus \{ q \})$, therefore it maps the generator to 0. We have obtained a contradiction, this proves that $i$ must be surjective.
\end{proof}

\begin{prop}
The map $i$ is injective.
\end{prop}

\begin{proof}
Since $i'$ and $i''$ are injective, we only need to prove that $i'(t\theta,x_1,y_1) \neq i''(x_2,y_2)$ if $t>0$ and $|x_2|+|y_2|<1$. Assume indirectly that $i'(t\theta,x_1,y_1) = i''(x_2,y_2)$.

Since $i''$ is continuous, injective, and $\tilde{V}$ is compact, $i''(\tilde{V})$ is the homeomorphic image of $\tilde{V}$. So it is a closed codimension-0 submanifold (with boundary) in $W''_0$, therefore any path in $W''_0$ connecting a point of $i''(\interior \tilde{V})$ with a point outside $i''(\tilde{V})$ must intersect $i''(\partial \tilde{V})$.

Let $\gamma : [t,1] \rightarrow T \big| _G \cap \overline{W'}$, $\gamma(s)=(s\theta,x_1,y_1)$.
 
Then $i'(\gamma([t,1]))$ is a path between $i'(\gamma(t)) = i'(t\theta,x_1,y_1) = i''(x_2,y_2) \in i''(\interior \tilde{V})$ and $i'(\gamma(1)) = i'(\theta,x_1,y_1)$. The latter is not in $i''(\tilde{V})$, because if $|x|+|y|=1$, then $i''(x,y) = i'(0,x,y) \neq i'(\theta,x_1,y_1)$ (because $i'$ is injective), and if $|x| + |y| < 1$, then $|a(\theta,x_1,y_1)| + |b(\theta,x_1,y_1)| = |x_1| + |y_1| = 1 > |x|+|y| > |c(x,y)|+|d(x,y)|$, so $i'(\theta, \allowbreak x_1,y_1) \neq i''(x,y)$. 

Therefore $i'(\gamma([t,1]))$ intersects $i''(\partial \tilde{V}) = i'(F_0 \cap \overline{W'})$. But $\gamma([t,1])$ is disjoint from $F_0 \cap \overline{W'}$, and $i'$ is injective, so this is a contradiction.

This proves that $i$ is injective.
\end{proof}

To sum up, $i$ is a continuous bijection between $T \big| _G \cap W$ and $T \big| _G \cap W''_0$, and since $T \big| _G \cap W$ is compact, this is a homeomorphism.

For each $p \in N$ we can construct this homeomorphism $i$. Together they will form a homeomorphism $I : T \big| _S \cap W \rightarrow W''_0$. (This $I$ is continuous, because for any $k$, $i$ is defined by the same formula for every $p \in U_k$, so $I$ is continuous restricted to $T \big| _{S|_{U_k}} \cap W$.) We saw that $i$ is identical on the boundary, so this holds for $I$ too.

This finishes the proof of the existence of a suitable $I$.

Since $i$ maps a subset of $T \big| _G \approx D^2 \times D^4$ into another one, it is homotopic to the inclusion $T \big| _G \cap W \hookrightarrow T \big| _G$ through a linear homotopy. The linear homotopy is well-defined, and can be defined for all fibers $G$ of $\S$, so $I$ is homotopic to the inclusion $T \big| _S \cap W \hookrightarrow T \big| _S$.

The same argument proves that $I$ composed with the homeomorphism $W''_0 \rightarrow W''_2$ is homotopic to the inclusion, hence the homeomorphism $W \rightarrow W_2$ we constructed is homotopic to the inclusion $W \hookrightarrow T$. This finishes the proof of Theorem \ref{thm:smooth}.

\begin{proof}[Proof of Proposition \ref{prop:or2}]
If $Y_1$ is both oriented and co-oriented, then it defines a local orientation of $X$ at each of its points (by taking the orientations of the tangent space and normal space of $Y_1$). Similarly $Y_2$ defines a local orientation of $X$ at each point of $Y_2$. So both of them determine local orientations of $X$ at points of $Y_1 \cap Y_2 = M$. This is a deformation retract of $T$, and $T$ is a codimension-0 submanifold in $X$, therefore the local orientations of $X$ along $M$ correspond to orientations of $T$.

First suppose that there exists an oriented and co-oriented $Y$. This determines an orientation of $T$. It must coincide with the orientation induced by $Y_i$ (because in the points of $B_i$, the orientation and co-orientation of $Y$ coincide with those of $Y_i$). This holds for both $Y_i$, so they must define the same orientation of $T$.

Next suppose that $Y_1$ and $Y_2$ define the same orientation of $T$. As we saw, we can construct a co-oriented $W \subset T$ such that $Y = ((Y_1 \cup Y_2) \setminus T) \cup W$ is co-oriented. \big[In Case 1 this follows from Proposition \ref{prop:or1}.\big] The orientation of $T$ and the co-orientation of $W$ determine an orientation of $W$ (by the condition that this orientation and the co-orientation of $W$ must induce the given orientation of $T$). Then $Y$ is oriented, because at $B_i$ the orientations of $Y_i$ and $W$ coincide (they satisfy the same condition). Therefore in this case there exists an oriented and co-oriented $Y$.
\end{proof}

\section{A lower bound for the number of components of $M$} 

Let $X$ be a closed $n$-manifold and $Y_1, Y_2$ connected oriented codimension-1 submanifolds in $X$ representing the integral homology classes $[Y_1], [Y_2] \in H_{n-1}(X; \Z)$. Suppose that $Y_1$ and $Y_2$ intersect each other transversally in $M$. Let $|M|$ denote the number of connected components of $M$. 

Let $t : H_{n-1}(X) \rightarrow H_{n-1}(X) / \Tor(H_{n-1}(X)) = F$ be the natural homomorphism. We call an element $\tilde{\alpha} \in F$ primitive if $\tilde{\alpha} = k \tilde{\beta}$ ($k \in \Z$, $\tilde{\beta} \in F$) implies $|k| = 1$. $F$ is a free Abelian group, so for each non-zero $\tilde{\alpha} \in F$ there is a unique positive integer $k$ and a primitive element $\tilde{\beta} \in F$ such that $\tilde{\alpha} = k \tilde{\beta}$, we will denote this $k$ by $\DIV(\tilde{\alpha})$ (and $\DIV(0) = 0$). For an $\alpha \in H_{n-1}(X)$ let $\DIV(\alpha) = \DIV(t(\alpha))$, $\alpha$ is called primitive if $\DIV(\alpha) = 1$.

If $X$ is orientable, then $H_{n-1}(X)$ is free, so for each $\alpha \in H_{n-1}(X)$ there is a primitive $\beta$ such that $\alpha = \DIV(\alpha) \beta$. If $X$ is non-orientable, then $\Tor(H_{n-1}(X)) = \Z_2$, let $\sigma$ denote the order-2 element. In this case for each $\alpha$ there is a primitive $\beta$ such that $\alpha = \DIV(\alpha) \beta$ or $\alpha = \DIV(\alpha) \beta + \sigma$. If $\DIV(\alpha)$ is odd, then an equation of the second type can be rewritten to one of the first type, because $\DIV(\alpha) \beta + \sigma = \DIV(\alpha) (\beta + \sigma)$, and $(\beta + \sigma)$ is primitive too.

The following theorem was proved in Meeks--Patrusky ([1]) and Meeks ([2]):

\begin{thm}
If an embedded oriented codimension-1 submanifold represents $\alpha \in H_{n-1}(X)$, then it has at least $C(\alpha)$ connected components, where
\[
C(\alpha) =
\begin{cases}
\DIV(\alpha) & \text{if $X$ is orientable} \\
\left \lfloor \frac{\DIV(\alpha)}{2} \right \rfloor & \text{if $X$ is non-orientable \linebreak and $\alpha = \DIV(\alpha) \beta$ for a primitive $\beta$} \\
\frac{\DIV(\alpha)}{2} + 1 & \text{if $X$ is non-orientable, $\DIV(\alpha)$ is even and $\alpha = \DIV(\alpha) \beta + \sigma$ } \\
& \text{for a primitive $\beta$}
\end{cases}
\]
\end{thm}

Using this theorem we can prove the following:

\begin{thm} \label{thm:lb1}
If $C([Y_1] + [Y_2]) \geq 3$, then $|M| \geq C([Y_1] + [Y_2]) - 1$. 
\end{thm}

\begin{proof}
We construct a graph $G$. The vertices of $G$ will be the connected components of $Y_1 \setminus M$ and $Y_2 \setminus M$. Each connected component $M'$ of $M$ will correspond to two edges. If the normal bundle of $M'$ in $Y_1$ is trivial, then it is contained in the boundaries of two connected components of $Y_1 \setminus M$, and $G$ will contain an edge between these two components (these components may be equal, so this edge may be a loop). If the normal bundle is non-trivial, then $M'$ is contained in the boundary of only one component, and $G$ will contain a loop on the corresponding vertex, such loops will be called special. Similarly, $M'$ determines an edge between the connected components of $Y_2 \setminus M$. The submanifolds $Y_1$ and $Y_2$ are connected, so $G$ has two connected components, $G_1$ and $G_2$, where $G_i$ contains the vertices corresponding to connected components of $Y_i \setminus M$.

We now prove that there are no vertices of degree $0$ in $G$. If $G_i$ contained such a vertex, then it would contain only this vertex (because it is connected) and no edges, so $M$ would be empty. Then $Y_1 \bigsqcup Y_2$ would represent $[Y_1] + [Y_2]$, so $C([Y_1] + [Y_2]) \leq 2$, which contradicts the theorem's assumption.

We apply the construction in the proof of Proposition \ref{prop:or1} to obtain a $Y$ representing $[Y_1] + [Y_2]$. We define another graph $G'$ with the property that its connected components are in a bijection with the connected components of $Y$. The vertices of $G'$ will be those of $G$. Again, edges will be defined for each connected component $M'$ of $M$, we will need to consider 3 cases:

First suppose that $M'$ corresponds to the non-special edges $x_1x_2$ and $y_1y_2$ in $G_1$ and $G_2$ (these may be loops). When we construct $Y$, we replace a neighbourhood of $M'$ in $Y_1 \cup Y_2$ by $W \cap T \big| _{M'}$, which joins $x_1$ with $y_1$ and $x_2$ with $y_2$, or $x_1$ with $y_2$ and $x_2$ with $y_1$ (more precisely, we should write here $x_1 \setminus T$ instead of $x_1$, etc.). The graph $G'$ will then contain the edges $x_1y_1$ and $x_2y_2$, or $x_1y_2$ and $x_2y_1$.

If $M'$ corresponds to a special loop on $x$ in $G_1$ and a non-special edge $y_1y_2$ in $G_2$, then $x \setminus T$, $y_1 \setminus T$ and $y_2 \setminus T$ will be in the same connected component of $Y$, so $G'$ will contain the edges $xy_1$ and $xy_2$. The case of a special loop in $G_2$ and a non-special edge in $G_1$ is similar.

If $M'$ corresponds to special loops on $x$ and $y$ in $G_1$ and $G_2$, then $x \setminus T$ and $y \setminus T$ will be in the same connected component of $Y$, so $G'$ will contain the edge $xy$.

The connected components of $G'$ correspond to the connected components of $Y$. By the theorem of Meeks, $Y$ has at least $C([Y_1] + [Y_2])$ connected components, so this is also true for $G'$. But $G'$ is a bipartite graph, and each of its vertices has degree at least $1$ (because this is true in $G$, and if a vertex is contained in an edge in $G$, then it will be contained in an edge in $G'$), so the number of connected components is at most the number of vertices in either $G_1$ or $G_2$. Since $G_i$ is connected in $G$, and contains $|M|$ edges, it contains at most $|M|+1$ vertices. This means that $C([Y_1] + [Y_2]) \leq |M|+1$. 
\end{proof}

The following generalisation also holds:

\begin{thm} \label{thm:lb2}
If $a, b \in \Z \setminus \{ 0 \}$ and $C(a [Y_1] + b [Y_2]) > |a| + |b|$, then 
\[
|M| \geq \frac{C(a [Y_1] + b [Y_2]) - \min(|a|,|b|)}{|ab|}.
\] 
\end{thm}

\begin{proof}
First we construct a $\tilde{Y}_1$ from $Y_1$. The normal bundle of $Y_1$ is an $\R^1$-bundle, so we can take its $S^0_{\varepsilon}$-bundle (this contains the points $\pm \varepsilon$ of each fiber). The new submanifold $\tilde{Y}_1$ will consist of $\left \lfloor \frac{|a|}{2} \right \rfloor$ copies of such an $S^0_{\varepsilon}$-bundle (ie.\ we construct this for $\left \lfloor \frac{|a|}{2} \right \rfloor$ different values of $\varepsilon$), and also $Y_1$ if $a$ is odd. The projection of an $S^0_{\varepsilon}$-bundle to $Y_1$ is a 2-fold cover, so this bundle has at most 2 connected components, so $\tilde{Y}_1$ has at most $|a|$ of them. If $a>0$ then they are given the same orientation as $Y_1$ (more precisely, the orientation of an $S^0_{\varepsilon}$-bundle is the pullback of the orientation of $Y_1$ by the projection), if $a < 0$, then $\tilde{Y}_1$ is given the opposite orientation. Thus the $\tilde{Y}_1$ we constructed represents $a [Y_1]$. We can construct similarly a $\tilde{Y}_2$ that represents $b [Y_2]$.

Let $\tilde{M} = \tilde{Y}_1 \cap \tilde{Y}_2$, it contains at most $|ab|$ connected components for each component of $M$, so $|\tilde{M}| \leq \left| ab \right| \left| M \right|$. The proof of Theorem \ref{thm:lb1} works for $\tilde{Y}_1$ and $\tilde{Y}_2$ with the following modifications:

The subgraphs $G_1$ and $G_2$ are not necessarily connected, but they contain at most $|a|$ and $|b|$ connected components respectively. The graph $G$ contains a vertex of degree 0 only if $M = \emptyset$, but in this case $\tilde{Y}_1 \bigsqcup \tilde{Y}_2$ represents $a [Y_1] + b [Y_2]$, and it has at most $|a| + |b|$ components, so $C(a [Y_1] + b [Y_2]) \leq |a| + |b|$, which contradicts our assumption. The number of vertices of $G_1$ is at most $|\tilde{M}|+|a|$, and the number of vertices of $G_2$ is at most $|\tilde{M}|+|b|$, therefore $G'$ has at most $|\tilde{M}| + \min(|a|,|b|)$ connected components. This means that $C(a [Y_1] + b [Y_2]) \leq \left| ab \right| \left| M \right| + \min(|a|,|b|)$.
\end{proof}

\section*{Bibliography}

\noindent
{[1]} W Meeks, J Patrusky, Representing codimension-one homology classes by embedded submanifolds, Pacific J. Math., vol 68 (1977) 175-176 \\
\\
{[2]} W H Meeks, III, Representing codimension-one homology classes on closed nonorientable manifolds by submanifolds, Illinois J. Math. 23 (1979) 199–210 \\
\\
{[3]} R. Thom, Quelques propri\'et\'es globales des vari\'et\'es diff\'erentiables, Comment. Math. Helv. 29 (1954) 17–85.

\end{document}